\documentclass[english]{article}
\usepackage[T1]{fontenc}
\usepackage[latin9]{inputenc}
\usepackage{color,soul}
\usepackage{framed}
\usepackage{amsmath}
\usepackage{amsthm}
\usepackage{amssymb}

\makeatletter

\providecommand{\tabularnewline}{\\}

\numberwithin{equation}{section}
\numberwithin{figure}{section}
\theoremstyle{plain}
\newtheorem{thm}{\protect\theoremname}
  \theoremstyle{remark}
  \newtheorem{claim}[thm]{\protect\claimname}

\usepackage{fullpage}\usepackage{makeidx}\usepackage{amsfonts}\usepackage{latexsym}

\theoremstyle{definition}
\newtheorem{lemma}{Lemma}[section]
\newtheorem{theorem}[lemma]{Theorem}\newtheorem{proposition}[lemma]{Proposition}\newtheorem{definition}[lemma]{Definition}\newtheorem{remark}[lemma]{Remark}\usepackage{times}

\usepackage{babel}

\usepackage[blocks]{authblk}
\title{Uniqueness of VOA structure of $3C$-algebra and $5A$-algebra}

\author{Chongying Dong\footnote{supported by the Simons foundation  634104}}
\affil{Department of Mathematics, University of
California, Santa Cruz, CA 95064 USA}
\author{Wen Zheng}
\affil{Department of Mathematics, University of
California, Santa Cruz, CA 95064 USA}

\usepackage{babel}

\usepackage{babel}

\usepackage{babel}

\usepackage{babel}
\providecommand{\claimname}{Claim}
\providecommand{\theoremname}{Theorem}

\usepackage{babel}
\providecommand{\claimname}{Claim}
\providecommand{\theoremname}{Theorem}

\numberwithin{equation}{section}

\usepackage{babel}
\providecommand{\claimname}{Claim}
\providecommand{\theoremname}{Theorem}

\makeatother
\usepackage{babel}
  \providecommand{\claimname}{Claim}
\providecommand{\theoremname}{Theorem}
\def\ha{\frac{1}{2}}
\begin{document}
\definecolor{shadecolor}{rgb}{0,1,0}
\maketitle
\begin{abstract}
 The structure of $3C$-algebra and $5A$-algebra constructed by Lam-Yamada-Yamauchi is studied  and the uniqueness of the vertex operator algebra structure of these two algebras is established. We also give the fusion rules for these two algebras.

\end{abstract}
\section{Introduction}
\sethlcolor{green}
The Monster simple group $\mathbb{M}$ \cite{G} is generated by some $2A$-involutions and the conjugacy class of the product of two $2A$-involutions $\tau\tau'$ is one of the nine classes $1A, 2A, 3A,
4A, 5A, 6A, 4B, 2B$ and $3C$ in  $\mathbb{M}$ \cite{C}. Moreover, each $2A$-involution $\tau$ defines a unique idempotent $e_{\tau}$ in the Monster Griess algebra, which is called an axis. The inner product of $e_{\tau}$ and $e_{\tau}'$ is uniquely determined by the conjugacy class of the product of two $2A$-involutions $\tau\tau'$.
From the construction of the moonshine vertex operator algebra $V^{\natural}$ \cite{FLM} we know that the Monster Griess algebra is the weight two subspace  $V^{\natural}_2$ of the $V^{\natural}.$ It was discovered in \cite{DMZ}
that  $V^{\natural}$ contains $48$ Virasoro vectors, each Virasoro vector generates a Virasoro vertex operator algebra isomorphic to $L(\ha,0)$ in $V^{\natural}$ and $L(\ha,0)^{\otimes 48}$ is a conformal subalgebra of $V^{\natural}.$ Such a Virasoro vector is called an Ising vector. Miyamoto later realized  that $2e_{\tau}$ is an Ising vector for any $2A$-involution $\tau.$  Thus there is a one-to-one correspondence
between $2A$-involutions of $\mathbb{M}$ and Ising vectors of $V^{\natural}$. According to a result in \cite{C},  the structure of the subalgebra generated
by two Ising vectors $e$ and $f$ in the algebra $V_{2}^{\natural}$ depends on only the conjugacy class of $\tau_{e}\tau_{f}.$ For the nine classes $1A, 2A, 3A,
4A, 5A, 6A, 4B, 2B$ and $3C,$ the inner product $\left\langle e,f\right\rangle$ are $\frac{1}{4}, \frac{1}{2^5}, \frac{13}{2^{10}}, \frac{1}{2^7}, \frac{3}{2^9}, \frac{5}{2^{10}}, \frac{1}{2^8}, 0, \frac{1}{2^8},$ respectively.

 It is natural to ask what the vertex operator subalgebra generated by two Ising vectors in an arbitrary vertex operator algebra is. A beautiful result given in \cite{S} asserts that the inner product of any two different Ising vectors again take these 9 values as in the case of the moonshine vertex operator algebra. In \cite{LYY}, \cite{LYY1}, for each of the nine cases, they constructed a subalgebra of $V_{\sqrt{2}E_{8}}$, which is generated by  two Ising vectors. These vertex operator algebras generated by two Ising vectors are simply called  $1A, 2A, 3A,
4A, 5A, 6A, 4B, 2B$ and $3C$-algebras, denoted by ${\cal U}_{nX}.$ But this raises two more questions: (1) Is the vertex operator algebra structure of these algebras constructed in \cite{LYY}, \cite{LYY1} unique? (2) Is any vertex operator algebra generated by two Ising vectors isomorphic to one of these 9 algebras? The uniqueness of VOA structure of ${\cal U}_{6A}$-algebra has been given in \cite{DJY}. In \cite{SY}, they proved the VOA generated by two Ising vectors whose inner product is $\frac{13}{2^{10}}$ has a unique VOA structure, so $3A$ case in question (2) has been solved thoroughly.

Now consider the uniqueness of VOA structure of ${\cal U}_{nX}$-algebra where $nX\not=3A$ or $6A$. Note that
$\mathcal{U}_{1A}\cong L\left(\frac{1}{2}, 0\right)$ and $\mathcal{U}_{2B}\cong L\left(\frac{1}{2}, 0\right)\otimes L\left(\frac{1}{2}, 0\right).$ So the uniqueness of VOA structure of these two algebras is trivial. We also know
that $\mathcal{U}_{2A}\cong L\left(\frac{1}{2}, 0\right)\otimes L\left(\frac{7}{10}, 0\right)\oplus L\left(\frac{1}{2}, \frac{1}{2}\right)\otimes L\left(\frac{7}{10}, \frac{3}{2}\right)$ is a simple current extension of the subVOA
$L\left(\frac{1}{2}, 0\right)\otimes L\left(\frac{7}{10}, 0\right)$, $\mathcal{U}_{4B}\cong L\left(\frac{1}{2}, 0\right)\otimes L\left(\frac{7}{10}, 0\right)\otimes L\left(\frac{7}{10}, 0\right) \oplus L\left(\frac{1}{2}, \frac{1}{2}\right)\otimes L\left(\frac{7}{10}, \frac{3}{2}\right)\otimes L\left(\frac{7}{10}, 0\right)\oplus L\left(\frac{1}{2}, \frac{1}{2}\right)\otimes L\left(\frac{7}{10}, 0\right)\otimes L\left(\frac{7}{10}, \frac{3}{2}\right)\oplus L\left(\frac{1}{2}, \frac{1}{2}\right)\otimes L\left(\frac{7}{10}, \frac{3}{2}\right)\otimes L\left(\frac{7}{10}, \frac{3}{2}\right)$ is a simple current extension of the subVOA $L\left(\frac{1}{2}, 0\right)\otimes L\left(\frac{7}{10}, 0\right)\otimes L\left(\frac{7}{10}, 0\right)$, so the uniqueness of the VOA structure of $\mathcal{U}_{2A}$ and $\mathcal{U}_{4B}$ follows from \cite{DM1}, also see Remark(\ref{simple current extension}).  $\mathcal{U}_{4A}\cong V_{\mathcal{N}}^{+}$,
So $\mathcal{U}_{4A}$, $\mathcal{U}_{5A}$ and $\mathcal{U}_{3C}$ are the three nontrivial cases left. In this paper, we only consider the uniqueness of VOA structure of $\mathcal{U}_{5A}$ and $\mathcal{U}_{3C}$-algebras. In fact, the uniqueness of VOA structure of $\mathcal{U}_{5A}$ and $\mathcal{U}_{3C}$-algebras can be deduced from Lemma C.1-C.3 in \cite{LYY}. In their paper, they mainly use associativity to prove these lemmas. Here we will give another proof by using commutativity and braiding matrix.

The paper is organized as follows. In Section 2, we review some basic
notions and some well known results in the vertex operator algebra
theory. In Section 3, we study the structure of the $5A$-algebra
and prove the uniqueness of the vertex operator algebra structure
on $\mathcal{U}$. In section 4, we study the structure of the $3C$-algebra
and prove the uniqueness of the vertex operator algebra structure
on $\mathcal{U}$. In section 5, we give fusion rules of the  $5A$-algebra and  $3C$-algebra.
\section{Preliminary}

In this section, we review some basics on vertex operators algebras. For more details of the definition of VOA and its all kinds of modules, etc, we refer \cite{FLM}, \cite{DLM1}, \cite{DLM2}.

\subsection{Ising vector}

\begin{definition} A vector $e\in V_{2}$ is called a \emph{conformal
vector with central charge $c_{e}$ }if it satisfies\emph{ $e_{1}e=2e$
}and $e_{3}e=\frac{c_{e}}{2}\mathbf{1}$. Then the operators $L_{n}^{e}:=e_{n+1},\ n\in\mathbb{Z}$,
satisfy the Virasoro commutation relation
\[
\left[L_{m}^{e},\ L_{n}^{e}\right]=\left(m-n\right)L_{m+n}^{e}+\delta_{m+n,\ 0}\frac{m^{3}-m}{12}c_{e}
\]
for $m,\ n\in\mathbb{Z}.$ A conformal vector $e\in V_{2}$ with central
charge $1/2$ is called an \emph{Ising vector }if $e$ generates the
Virasoro vertex operator algebra $L(1/2,\ 0)$.

\end{definition}

\subsection{Invariant bilinear form}
Let $M=\oplus_{\lambda_{\in\mathbb{C}}}M_{\lambda}$ be a $V$-module.
The restricted dual of $M$ is define by $M'=\oplus_{\lambda_{\in\mathbb{C}}}M_{\lambda}^{\ast}$
where $M_{\lambda}^{\ast}=\text{Hom}_{\mathbb{C}}\left(M_{\lambda},\mathbb{C}\right).$
It was proved in \cite{FHL} that $M'=\left(M',Y_{M'}\right)$ is
naturally a $V$-module such that
\[
\left\langle Y_{M'}\left(v,z\right)f,u\right\rangle =\left\langle f,Y_{M}\left(e^{zL\left(1\right)}\left(-z^{-2}\right)^{L\left(0\right)}v,z^{-1}\right)u\right\rangle ,
\]
for $v\in V,$ $f\in M'$ and $u\in M$, and $\left(M'\right)'\cong M$.
Moreover, if $M$ is irreducible, so is $M'$. A $V$-module $M$
is said to be \emph{self dual} if $M\cong M'$.

\begin{proposition} \label{extension property} \cite{FHL} For $v , v_{1}, v_{2}\in V$ and $v'\in V'$
, we have the following equality of rational functions:

(a):\textbf{Commutativity}:
\begin{equation}
 \iota_{12}^{-1}\left\langle v',Y\left(v_{1},z_{1}\right)Y\left(v_{2},z_{2}\right)w\right\rangle =\iota_{21}^{-1}\left\langle v',Y\left(v_{2},z_{2}\right)Y\left(v_{1},z_{1}\right)w\right\rangle \label{communitivity}
\end{equation}

(b):\textbf{Associativity}:
\begin{equation}
\iota_{12}^{-1}\left\langle v',Y\left(v_{1},z_{1}\right)Y\left(v_{2},z_{2}\right)w\right\rangle =\left(\iota_{20}^{-1}\left\langle t,Y\left(Y\left(v_{1},z_{0}\right)v_{2},z_{2}\right)w\right\rangle \right)|_{z_{0}=z_{1}-z_{2}}\label{associavity}
\end{equation}
where $\iota_{12}^{-1}f\left(z_{1},z_{2}\right)$ denotes the formal
power expansion of an analytic function $f\left(z_{1},z_{2}\right)$
in the domain $\left|z_{1}\right|>\left|z_{2}\right|$ . \end{proposition}

The following result about bilinear form on $V$ is from \cite{L}:
\begin{theorem}\label{bilinear form}The space of invariant bilinear
forms on $V$ is isomorphic to the space
\[
\left(V_{0}/L\left(1\right)V_{1}\right)^{*}=\mbox{Hom}_{\mathbb{C}}\left(V_{0}/L\left(1\right)V_{1},\mathbb{C}\right).
\]
 \end{theorem}

\subsection{Intertwining operator and fusion rules}

\begin{definition} Let $\left(V,Y\right)$ be a vertex operator algebra
and let $\left(M^{i},Y^{i}\right),\ \left(M^{j},Y^{j}\right)$ and
$\left(M^{k},Y^{k}\right)$ be three $V$-modules. An \emph{intertwining
operator of type $\left(\begin{array}{c}
M^{k}\\
M^{i}\ M^{j}
\end{array}\right)$} is a linear map
\begin{gather*}
\mathcal{Y}\left(\cdot,z\right):\ M^{i}\to\text{\ensuremath{\mbox{Hom}\left(M^{j},\ M^{k}\right)\left\{ z\right\} }}\\
u\to\mathcal{Y}\left(u,z\right)=\sum_{n\in\mathbb{Q}}u_{n}z^{-n-1}
\end{gather*}
satisfying:

(1) for any $u\in M^{i}$ and $v\in M^{j}$, $u_{n}v=0$ for $n$
sufficiently large;

(2) $\mathcal{Y}(L_{-1}v,\ z)=\left(\frac{d}{dz}\right)\mathcal{Y}\left(v,z\right)$
for $v\in M^{i}$;

(3) (Jacobi Identity) for any $u\in V,\ v\in M^{i}$,
\begin{alignat*}{1}
 & z_{0}^{-1}\delta\left(\frac{z_{1}-z_{2}}{z_{0}}\right)Y^{k}\left(u,z_{1}\right)\mathcal{Y}\left(v,z_{2}\right)-z_{0}^{-1}\delta\left(\frac{-z_{2}+z_{1}}{z_{0}}\right)\mathcal{Y}\left(v,z_{2}\right)Y^{j}\left(u,z_{1}\right)\\
 & =z_{2}^{-1}\left(\frac{z_{1}-z_{0}}{z_{2}}\right)\mathcal{Y}\left(Y^{i}\left(u,z_{0}\right)v,z_{2}\right).
\end{alignat*}
The space of all intertwining operators of type $\left(\begin{array}{c}
M^{k}\\
M^{i}\ M^{j}
\end{array}\right)$ is denoted $I_{V}\left(\begin{array}{c}
M^{k}\\
M^{i}\ M^{j}
\end{array}\right)$. Without confusion, we also denote it by $I_{i,j}^{k}.$ Let $N_{i,\ j}^{k}=\dim I_{i,j}^{k}$.
These integers $N_{i,j}^{k}$ are called the \emph{fusion rules}.
\end{definition}

\begin{definition}

Let $M^{1}$ and $M^{2}$ be $V$-modules. A tensor product for the
ordered pair $\left(M^{1},M^{2}\right)$ is a pair $\left(M,\mathcal{Y}\left(\cdot,z\right)\right)$
which consists of a $V$-module $M$ and an intertwining operator
$\mathcal{Y}\left(\cdot,z\right)$ of type $\left(\begin{array}{c}
W\\
M^{1}\ M^{2}
\end{array}\right)$
satisfies the following universal property: For any $V$-module
$W$ and any intertwining operator $\mathcal{I}\left(\cdot,z\right)$ of type
$\left(\begin{array}{c}
W\\
M^{1}\ M^{2}
\end{array}\right)$, there exists a unique $V$-homomorphism
$\phi$ from $M$ to $W$ such that $\mathcal{I}\left(\cdot,z\right)=\phi\circ \mathcal{Y}\left(\cdot,z\right).$
From the definition it is easy to see that if a tensor product of $M^{1}$
and $M^{2}$ exists, it is unique up to isomorphism. In this case,
we denote the \emph{fusion product} by $M^{1}\boxtimes_{V}M^{2}.$

\end{definition}
Let $V^{1}$ and $V^{2}$ be vertex operator algebras. Let $M^{i}$
, $i=1,2,3$, be $V^{1}$-modules, and $N^{i}$, $i=1,2,3$, be $V^{2}$-modules.
Then $M^{i}\otimes N^{i}$, $i=1,2,3$, are $V^{1}\otimes V^{2}$-modules
by \cite{FHL}. The following property was given in \cite{ADL}:

\begin{proposition} \label{fusion of tensor product}If $N_{M^{1},M^{2}}^{M^{3}}<\infty$
or $N_{N^{1},N^{2}}^{N^{3}}<\infty,$ then
\[
N_{M^{1}\otimes N^{1},M^{2}\otimes N^{2}}^{M^{3}\otimes N^{3}}=N_{M^{1},M^{2}}^{M^{3}}N_{N^{1},N^{2}}^{N^{3}}.
\]
\end{proposition}

\subsection{Simple current extensions}
\begin{definition}
A VOA is graded by an ablian group G if $V=\oplus_{g\in G}V^{g},$ and $u_{n}v\in V^{g+h}$ for any $u\in V^{g}, v\in V^{h}$, and
$n\in\mathbb{Z} .$
\end{definition}
The following proposition was given in \cite{DM1}:
\begin{proposition}\label{simple current}
Let $V=\sum_{g \in G}V^{g}$ be a simple $G$-graded VOA such that $V^{g}\not=0$ for all $g\in G$ and $N_{V^{g},V^{h}}^{V^{k}}=\delta_{g+h,k}$ for all $g,h,k\in G$. Then the VOA structure of $V$ is determined uniquely by the $V^{0}$-module structure of $V$.
\end{proposition}

\begin{definition} \
Let $V$ be a simple VOA. An irreducible $V$ module $U$ is called a simple current $V$-module if it satisfies that for every irreducible $V$-module
$M$, the fusion product $U\boxtimes M$ is also irreducible.
\end{definition}

\begin{definition}
Let $V=\sum_{g \in G}V^{g}$ be a simple $G$-graded VOA such that $V^{g}\not=0$ for all $g\in G$, then $V$ is called a $G$-graded simple current extension if all $V^{g},g\in G$ are simple current $V^{0}$-module.
\end{definition}

\begin{remark}\label{simple current extension}
Let $V=\sum_{g \in G}V^{g}$ be a $G$-graded simple current extension. Then by proposition \ref{simple current}, we have the VOA structure of $V$ is determined uniquely by the $V^{0}$-module structure of $V$.
\end{remark}

\subsection{Quantum Galois Theory}

\begin{theorem} \label{classical galois theory}\cite{DM2}Suppose that $V$
is a simple vertex operator algebra and that $G$ is a finite and
faithful solvable group of automorphisms of $V$. Then the following
hold:

(i)$V=\oplus_{\chi\in\text{Irr}\left(G\right)}V^{\chi},$ where $V^{\chi}$ is the subspace of $V$ on which $G$ acts according to
the character $\chi$. Each $V^{\chi}$ is nonzero;

(ii) For $\chi\in\text{Irr}\left(G\right)$, each $V^{\chi}$ is a
simple module for the $G$-graded vertex operator algebra $\mathbb{C}G\otimes V^{G}$
of the form

\[
V^{\chi}=M_{\chi}\otimes V_{\chi},
\]
where $M_{\chi}$ is the simple $\mathbb{C}G$-module affording $\chi$
and where $V_{\chi}$ is a simple $V^{G}$-module.

(iii) The map $M_{\chi}\mapsto V_{\chi}$ is a bijection from the
set of simple $\mathbb{C}G$-modules to the set of (inequivalent)
simple $V^{G}$-modules which are contained in $V$.

\end{theorem}

\begin{definition} \cite{DJX}
 Let $V$ be a vertex operator algebra of CFT type and $M$ a $V$-module such that the trace function $Z_{V}\left(\tau\right)$ and $Z_{M}\left(\tau\right)$
exist. The quantum dimension of $M$ over $V$ is defined as
\[
 q\dim_{V}M=\lim_{y\to 0}\frac{Z_{M}\left(iy\right)}{Z_{V}\left(iy\right)}
\]
Alternatively, we can use the following definiton:
\[
q\dim_{V}M=\lim_{q\to 1^{-}}\frac{ch_{q}M}{ch_{q}V}
\]
\end{definition}

\begin{theorem} (theorem 6.3 in \cite{DJX})\label{quantum dimension and orbifold module} Let
$V$ be a rational and $C_{2}$-cofinite simple vertex operator algebra.
Assume $V$ is $g$-rational and the conformal weight of any irreducible
$g$-twisted $V$-module is positive except for $V$ itself for all
$g\in G$. Then
\[
q\dim_{V^{G}}V_{\chi}=\dim W_{\chi}.
\]
\end{theorem}

\begin{proposition} \cite{DJX} \label{qdim of simple current}
Let $V$ be a
rational and $C_{2}$-cofinite simple vertex operator algebra of CFT
type with $V\cong V'$. Let $M^{0},\ M^{1},\ \cdots,\ M^{d}$ be all
the inequivalent irreducible $V$-modules with $M^{0}\cong V$ . The
corresponding conformal weights $\lambda_{i}$ satisfy $\lambda_{i}>0$
for $0<i\le d$. Then

(i) $q\dim_{V}\left(M^{i}\boxtimes M^{j}\right)=q\dim_{V}M^{i}\cdot q\dim_{V}M^{j},$
$\forall i,j$.

(ii) A $V$-module $M^{i}$ is a simple current if and only if $q\dim_{V}M^{i}=1$.

(iii) $q\dim_{V}M^{i}\in\left\{ 2\cos\left(\pi/n\right)|n\ge3\right\} \cup\left\{ a|2\le a<\infty,a\ is\ algebraic\right\} .$

\end{proposition}

\begin{remark} \cite{ADJR} \label{product property of qdim}Let $U$
and $V$ be a vertex operator algebra under the same assumption with Proposition \ref{qdim of simple current}, $M$ be a $U$-module
and $N$ be a $V$-module. Then

\[
q\dim_{U\otimes V}M\otimes N=q\dim_{U}M\cdot q\dim_{U}N.
\]
\end{remark}

\subsection{\label{subsec:The-unitary-series}The unitary series of the Virasoro
VOAs}
From now on we always assume $p,q\in \{2,3,4,...\}$, and $p,q$ are relatively prime.
\begin{definition} \cite{W} \label{admissiblepair} An ordered triple of pairs of integers $\left(\left(m,n\right),\left(m',n'\right),\left(m'',n''\right)\right)$
is  \emph{admissible }if $0< m,m',m''< p,0< n,n',n''< q$,
$m+m'+m''<2p,$ $n+n'+n''<2q$, $m<m'+m''$, $m'<m+m''$,
$m''<m+m'$, $n<n'+n''$, $n'<n+n''$, $n''<n+n'$, and the sums $m+m'+m''$,
$n+n'+n''$ are odd. We identify the triples $\left(\left(m,n\right),\left(m',n'\right),\left(m'',n''\right)\right)$ and
$\left(\left(m,n\right),\left(p-m',q-n'\right),\left(p-m'',q-n''\right)\right)$.
\end{definition}

Let $c_{p,q}=1-\frac{6(p-q)^2}{pq}$, $h_{m,n}=\frac{(np-mq)^2-(p-q)^2}{4pq}$, $0<m<p, 0<n<q$. $L\left(c_{p,q},h_{m,n}\right)$ is
the irreducible highest weight representation of the Virasoro algebra $L$ with highest weight $\left(c_{p,q},h_{m,n}\right)$. Then $L\left(c_{p,q},0\right)$
has a VOA structure. Moreover, we have the following results \cite{W}:

\begin{theorem} \cite{W}
The vertex operator algebra $L\left(c_{p,q},0\right)$ is rational and the minimal modules $L\left(c_{p,q},h_{m,n}\right)$, $0<m<p, 0<n<q$ are
all irreducible representations of $L\left(c_{p,q},0\right)$.
\end{theorem}

\begin{theorem} \label{fusion rules of virasoro modules}\cite{W}
The fusion rules between $L\left(c_{p,q},0\right)$-modules $L\left(c_{p,q},h_{m',n'}\right)$ and $L\left(c_{p,q},h_{m'',n''}\right)$ are
\[
L\left(c_{p,q},h_{m',n'}\right)\boxtimes L\left(c_{p,q},h_{m'',n''}\right) =\sum_{\left(m,n\right)}N_{\left(m',n'\right),\left(m'',n''\right)}^{\left(m,n\right)}L\left(c_{p,q},h_{m,n}\right),
\]
where $N_{\left(m',n'\right),\left(m'',n''\right)}^{\left(m,n\right)}$
is $1$ iff $\left(\left(m,n\right),\left(m',n'\right),\left(m'',n''\right)\right)$
is an admissible triple of pairs and $0$ otherwise. \end{theorem}

\subsection{\label{subsec:Braiding-matrix}Braiding matrices}
Now let's recall four point functions. Let $V$ be a rational VOA and $\mathcal{A}=\{M^1, M^2,..., M^m\}$ a set of representatives of equaivalence
class of irreducible $V$-modules. Then we have the following result:
\begin{theorem}\label{huangyizhi1}\cite{H1}
For any modules $M^{a_{1}}, M^{a_{2}}, M^{a_{3}}, M^{a_{4}}$ and $M^{a_{5}}$ and any intertwining operators $\mathcal{Y}_{a_{1},a_{5}}^{a_{4}} \in I_{a_{1}, a_{5}}^{a_{4}}$ and $\mathcal{Y}_{a_{2},a_{3}}^{a_{5}} \in I_{a_{2}, a_{3}}^{a_{5}}$,  there exist $\mathcal{Y}_{a_{2},\mu;i}^{a_{4}}\in I_{a_{2},\mu}^{a_{4}}$ and $\mathcal{Y}_{a_{1},a_{3};i}^{\mu}\in I_{a_{1}, a_{3}}^{\mu}$ for all $M^{\mu} \in \mathcal{A}$ and $i=1, 2, ..., l$
for some $l$
 such that the multi-valued
analytic function
\begin{alignat}{1}
\left\langle u_{a_{4}}',\mathcal{Y}_{a_{1},a_{5}}^{a_{4}}\left(u_{a_{1}},x_{1}\right)\mathcal{Y}_{a_{2},a_{3}}^{a_{5}}\left(u_{a_{2}},x_{2}\right)u_{a_{3}}\right\rangle|_{x_{1}=z_{1}, x_{2}=z_{2}}  \label{Huang1}
\end{alignat}
of $z_{1}$ and $z_{2}$ in the region $|z_{1}|>|z_{2}|>0$ and the multi-valued analytic function
\begin{alignat}{1}
\sum_{\mu}\sum_{i=1}^{l}\left\langle u_{a_{4}}',\mathcal{Y}_{a_{2},\mu;i}^{a_{4}}\left(u_{a_{2}},x_{2}\right)\mathcal{Y}_{a_{1},a_{3};i}^{\mu}\left(u_{a_{1}},x_{1}\right)u_{a_{3}}\right\rangle|_{x_{1}=z_{1}, x_{2}=z_{2}} \label{Huang2}
\end{alignat}
of $z_{1}$ and $z_{2}$ in the region $|z_{2}|>|z_{1}|>0$, are analytic extensions of each other.

\end{theorem}

\begin{remark}\label{huangyizhi2}
For the multi-valued analytic functions in (\ref{Huang1}), (\ref{Huang2}), we can get a single valued branch which is called \emph{prefered branch} in \cite{H2}
and we still use the same notation as (\ref{Huang1}), (\ref{Huang2}). Then we can get a similar result as proposition \ref{extension property}:
\begin{alignat}{1}
& \iota_{12}^{-1}\left\langle u_{a_{4}}',\mathcal{Y}_{a_{1},a_{5}}^{a_{4}}\left(u_{a_{1}},x_{1}\right)\mathcal{Y}_{a_{2},a_{3}}^{a_{5}}\left(u_{a_{2}},x_{2}\right)u_{a_{3}}\right\rangle|_{x_{1}=z_{1}, x_{2}=z_{2}} \nonumber\\
&= \iota_{21}^{-1}\sum_{\mu}\sum_{i=1}^{l}\left\langle u_{a_{4}}',\mathcal{Y}_{a_{2},\mu;i}^{a_{4}}\left(u_{a_{2}},x_{2}\right)\mathcal{Y}_{a_{1},a_{3};i}^{\mu}\left(u_{a_{1}},x_{1}\right)u_{a_{3}}\right\rangle|_{x_{1}=z_{1}, x_{2}=z_{2}} \nonumber
\end{alignat}
we simply denote it as
\begin{alignat}{1}
& \iota_{12}^{-1}\left\langle u_{a_{4}}',\mathcal{Y}_{a_{1},a_{5}}^{a_{4}}\left(u_{a_{1}},z_{1}\right)\mathcal{Y}_{a_{2},a_{3}}^{a_{5}}\left(u_{a_{2}},z_{2}\right)u_{a_{3}}\right\rangle \nonumber\\
&= \iota_{21}^{-1}\sum_{\mu}\sum_{i=1}^{l}\left\langle u_{a_{4}}',\mathcal{Y}_{a_{2},\mu;i}^{a_{4}}\left(u_{a_{2}},z_{2}\right)\mathcal{Y}_{a_{1},a_{3};i}^{\mu}\left(u_{a_{1}},z_{1}\right)u_{a_{3}}\right\rangle \nonumber.
\end{alignat}
\end{remark}

Let $\left\{ \mathcal{Y}_{a,b;i}^{c}|i=1,\cdots,N_{a,b}^{c}\right\} $
be a basis of $I_{a,b}^{c}$. From \cite{H2},

\[
\iota_{12}^{-1}\left\langle u_{a_{4}'},\mathcal{Y}_{a_{1},a_{5};i}^{a_{4}}\left(u_{a_{1}},z_{1}\right)\mathcal{Y}_{a_{2},a_{3};j}^{a_{5}}\left(u_{a_{2}},z_{2}\right)u_{a_{3}}|i=1,\cdots,N_{a_{1},a_{5}}^{a_{4}},j=1,\cdots,N_{a_{2},a_{3}}^{a_{5}},\forall a_{5}\right\rangle
\]
is a linearly independent set. Fix a basis of intertwining operators.
Then by remark \ref{huangyizhi2} there exists $\left(B_{a_{4,}a_{1}}^{a_{3},a_{2}}\right)_{\mu,\gamma}^{i,j;k,l}\in\mathbb{C}$
such that
\begin{alignat}{1}
 & \iota_{12}^{-1}\left\langle u_{a_{4}'},\mathcal{Y}_{a_{3},\mu;i}^{a_{4}}\left(u_{a_{3}},z_{1}\right)\mathcal{Y}_{a_{2},a_{1};j}^{\mu}\left(u_{a_{2}},z_{2}\right)u_{a_{1}}\right\rangle \nonumber \\
 & =\sum_{k,l,\gamma}\left(B_{{a_{4},a_{1}}}^{a_{3},a_{2}}\right)_{\mu,\gamma}^{i,j;k,l} \iota_{21}^{-1}\left\langle u_{a_{4}'},\mathcal{Y}_{a_{2},\gamma;k}^{a_{4}}\left(u_{a_{2}},z_{2}\right)\mathcal{Y}_{a_{3},a_{1};l}^{\gamma}\left(u_{a_{3}},z_{1}\right)u_{a_{1}}\right\rangle \label{braiding matrix property}
\end{alignat}
(see \cite{H2}). $B_{{a_{4},a_{1}}}^{a_{3},a_{2}}$ is called
the\emph{ braiding matrix}.

Now let's recall some formulas about minimal models of Virasoro vertex
operator algebra given in \cite{FFK}.

Let $\alpha_{-}^{2}=\frac{p}{p'}$, $\alpha_{+}^{2}=\frac{p'}{p}$, here $p'=p+1$.
$x=\exp\left(2\pi i\alpha_{+}^{2}\right)$, $y=\exp\left(2\pi i\alpha_{-}^{2}\right)$,
$\left[l\right]=x^{l/2}-x^{-l/2}$, $\left[l\right]'=y{}^{l/2}-y{}^{-l/2}.$ Denote $c_{p}=1-\frac{6}{p\left(p+1\right)}$ with $p=2,3,4,\cdots$, $h_{\left(i'\ i\right)}^{\left(p\right)}=\frac{1}{4}\left(i'^{2}-1\right)\alpha_{-}^{2}-\frac{1}{2}\left(i'i-1\right)+\frac{1}{4}\left(i^{2}-1\right)\alpha_{-}^{-2}=\frac{\left(pi'-\left(p+1\right)i\right)^{2}-1}{4p\left(p+1\right)}$.
Now we fix central charge $c_{p}$, denote $L\left(c_{p},h_{\left(i',i\right)}^{\left(p\right)}\right)$
by $\left(i',i\right)$.  Note that here $\left(i',i\right)$ is the same as $\left(i,i'\right)$ in Theorem \ref{fusion rules of virasoro modules}. Let $\left(a',a\right)$, $\left(m',m\right)$,
$\left(n',n\right)$, $\left(c',c\right)$, $\left(b',b\right)$,
$\left(d',d\right)$ be irreducible $L\left(c_{p},0\right)$-modules,
the braiding matrices of screened vertex operators have the almost
factorized form (cf. (2.19) of \cite{FFK}):
\begin{alignat}{1}
 & \left(B_{\left(m',m\right),\left(n',n\right)}^{\left(a',a\right),\left(c',c\right)}\right){}_{\left(b',b\right),\left(d',d\right)}\nonumber \\
 & =i^{-\left(m'-1\right)\left(n-1\right)-\left(n'-1\right)\left(m-1\right)}\left(-1\right)^{1/2\left(a-b+c-d\right)\left(n'+m\right)+1/2\left(a'-b'+c'-d'\right)\left(n+m\right)}\label{FFK 2.19}\\
 & \cdot r'\left(a',m',n',c'\right)_{b',d'}\cdot r\left(a,m,n,c\right)_{b,d},\nonumber
\end{alignat}
where the nonvanishing matrix elements of $r$-matrices are

\begin{gather}
r\left(a,1,n,c\right)_{a,c}=r\left(a,m,1,c\right)_{c,a}=1,\nonumber \\
r\left(l\pm2,2,2,l\right)_{l\pm1,l\pm1}=x^{1/4},\nonumber \\
r\left(l,2,2,l\right)_{l\pm1,l\pm1}=\mp x^{-1/4\mp l/2}\frac{\left[1\right]}{\left[l\right]},\nonumber \\
r\left(l,2,2,l\right)_{l\pm1,l\mp1}=x^{-1/4}\frac{\left[l\pm1\right]}{\left[l\right]},\label{FFK 2.20}
\end{gather}
and the other $r$-matrices are given by the recursive relation
\begin{gather}
r\left(a,m+1,n,c\right)_{b,d}=\sum_{d_{1}\ge1}r\left(a,2,n,d_{1}\right)_{a_{1},d}\cdot r\left(a_{1},m,n,c\right)_{b,d_{1}},\nonumber \\
r\left(a,m,n+1,c\right)_{b,d}=\sum_{d_{1}\ge1}r\left(a,m,2,c_{1}\right)_{b,d_{1}}\cdot r\left(d_{1},m,n,c\right)_{c_{1},d},\label{FFK 2.21}
\end{gather}
for any choice of $a_{1}$ and $c_{1}$ compatible with the fusion
rules. The $r'$ matrices are given by the same formulas with the
replacement $x\to y$, $\left[\ \ \ \ \right]\to\left[\ \ \ \ \right]'.$

\begin{remark}\label{A5C3pp'} Use the above notation, we see that the central charge of the model $L\left(\frac{25}{28},0\right)$
corresponds to the parameter $\alpha_{-}^{2}=\frac{7}{8}$ with $p=7,$
$p'=8$. The pairs $\left(1,1\right)$, $\left(3,1\right)$
, $\left(5,1\right)$ and $\left(7,1\right)$ correspond to the highest weights 0, $\frac{3}{4}$, $\frac{13}{4}$
and $\frac{15}{2}$ respectively.  The central charge of the model $L\left(\frac{21}{22},0\right)$
corresponds to the parameter $\alpha_{-}^{2}=\frac{11}{12}$ with $p=11,$
$p'=12$. The pairs $\left(1,1\right)$, and $\left(7,1\right)$
correspond to the highest weights 0 and $8$ respectively.
\end{remark}

Now we will prove two lemmas which will be used in the proof of uniqueness of $5A$ and $3C$ algebra.

First we consider braiding matrix for $L\left(\frac{25}{28},0\right)$-modules.
$ P_{2}=L\left(\frac{25}{28},\frac{15}{2}\right)$, $ P_{3}=L\left(\frac{25}{28},\frac{3}{4}\right)$,$ P_{4}=L\left(\frac{25}{28},\frac{13}{4}\right)$ are
irreducible $L\left(\frac{25}{28},0\right)$-modules. For convenience,
we will denote $\left(B_{P_{a},P_{b}}^{P_{c},P_{d}}\right)_{P_{e},P_{f}}$
by $\left(B_{a,b}^{c,d}\right)_{e,f}$, $a,b,c$, $d,$ $e$,
$f$ $\in\left\{ 2,3,4\right\} $.
\begin{lemma} \label{5Anonzero}$\left(B_{3,3}^{4,4}\right)_{4,4}\cdot\left(B_{3,3}^{4,4}\right)_{2,3}-\left(B_{3,3}^{4,4}\right)_{4,3}\cdot\left(B_{3,3}^{4,4}\right)_{2,4}\not=0$, and $\left(B_{3,3}^{4,4}\right)_{3,2}\cdot\left(B_{3,3}^{4,4}\right)_{4,4}-\left(B_{3,3}^{4,4}\right)_{4,2}\cdot\left(B_{3,3}^{4,4}\right)_{3,4}\not=0$. \end{lemma}
\begin{proof} Using (\ref{FFK 2.19}), (\ref{FFK 2.20}) and Remark \ref{A5C3pp'} , we have $\left(B_{3,3}^{4,4}\right)_{4,4}=r'\left(5,3,3,5\right)_{5,5}$, Using
(\ref{FFK 2.20}) and (\ref{FFK 2.21}) we obtain:

\[
r'\left(5,3,3,5\right)_{5,5}=r'\left(5,2,3,4\right)_{4,5}\cdot r'\left(4,2,3,5\right)_{5,4}+r'\left(5,2,3,6\right)_{4,5}\cdot r'\left(4,2,3,5\right)_{5,6}
\]

with
\begin{alignat*}{1}
r'\left(5,2,3,4\right)_{4,5} & =r'\left(5,2,2,3\right)_{4,4}\cdot r'\left(4,2,2,4\right)_{3,5}=y^{\frac{1}{4}}\cdot y^{-\frac{1}{4}} \frac{\left[3\right]'}{\left[4\right]'}= \frac{\left[3\right]'}{\left[4\right]'},
\end{alignat*}
\begin{alignat*}{1}
r'\left(4,2,3,5\right)_{5,4} & =r'\left(4,2,2,6\right)_{5,5}\cdot r'\left(5,2,2,5\right)_{6,4}=y^{\frac{1}{4}}\cdot y^{-\frac{1}{4}} \frac{\left[6\right]'}{\left[5\right]'}=\frac{\left[6\right]'}{\left[5\right]'},
\end{alignat*}
\begin{alignat*}{1}
r'\left(5,2,3,6\right)_{4,5} & =r'\left(5,2,2,5\right)_{4,4}\cdot r'\left(4,2,2,6\right)_{5,5}+r'\left(5,2,2,5\right)_{4,6}\cdot r'\left(6,2,2,6\right)_{5,5}\\
 & =y^{5/2}\frac{\left[1\right]'\left[6\right]'+\left[4\right]'\left[1\right]'}{\left[5\right]'\left[6\right]'},
\end{alignat*}

\begin{alignat*}{1}
r'\left(4,2,3,5\right)_{5,6}=r'\left(4,2,2,4\right)_{5,5}\cdot r'\left(5,2,2,5\right)_{4,6}=-y^{-5/2}\frac{\left[1\right]'}{\left[5\right]'},
\end{alignat*}

where $\left[l\right]'=2i\sin\left(\frac{7}{8}\pi l\right)$, $y=\exp\left(\frac{7}{4}\pi i\right)$.
Then we have:

\begin{alignat}{1}
 \left(B_{3,3}^{4,4}\right)_{4,4}= r'\left(5,3,3,5\right)_{5,5}=\frac{\left[6\right]'\left[3\right]'}{\left[5\right]'\left[4\right]'}-\frac{\left[1\right]'^2\left(\left[4\right]'+\left[6\right]'\right)}{\left[5\right]'^2\left[6\right]'}\label{A5B334444}
\end{alignat}

Similarly, we can obtain:
\begin{alignat}{1}
 \left(B_{3,3}^{4,4}\right)_{2,3}= r'\left(5,3,3,5\right)_{7,3}=y\frac{\left[6\right]'\left[7\right]'}{\left[4\right]'\left[5\right]'}\label{A5B334423}
\end{alignat}

\begin{alignat}{1}
 \left(B_{3,3}^{4,4}\right)_{4,3}= r'\left(5,3,3,5\right)_{5,3}=y^2\frac{\left[1\right]'\left[6\right]'}{\left[4\right]'\left[5\right]'}\label{A5B334443}
\end{alignat}

\begin{alignat}{1}
 \left(B_{3,3}^{4,4}\right)_{2,4}= r'\left(5,3,3,5\right)_{7,5}=-y^{-3}\left(\frac{\left[1\right]'\left[7\right]'\left(\left[4\right]'+\left[6\right]' \right)}{\left[5\right]'^2\left[4\right]'}+\frac{\left[1\right]'\left[7\right]'\left(\left[5\right]'+\left[7\right]' \right)}{\left[6\right]'^2\left[5\right]'}\right)\label{A5B334424}
\end{alignat}
From (\ref{A5B334444}),  (\ref{A5B334423}),  (\ref{A5B334443}),  (\ref{A5B334424}) and a direct computation, we can obtain:

\begin{alignat}{1}
\left(B_{3,3}^{4,4}\right)_{4,4}\cdot\left(B_{3,3}^{4,4}\right)_{2,3}-\left(B_{3,3}^{4,4}\right)_{4,3}\cdot\left(B_{3,3}^{4,4}\right)_{2,4}
&=\frac{\sqrt{2}-1}{2}+\frac{3-2\sqrt{2}}{2}i\not=0. \label{A5nonzero1}
\end{alignat}

By a similar process, we can obtain:
\begin{alignat}{1}
\left(B_{3,3}^{4,4}\right)_{3,2}\cdot\left(B_{3,3}^{4,4}\right)_{4,4}
&=y^{-1}\frac{\left[4\right]'\left[3\right]'}{\left[6\right]'\left[5\right]'}\left(\frac{\left[6\right]'\left[3\right]'}{\left[5\right]'\left[4\right]'}-\frac{\left[1\right]'^2\left(\left[4\right]'+\left[6\right]'\right)}{\left[5\right]'^2\left[6\right]'} \right) \nonumber \\
&=y^{-1}\left(\sqrt{2}-1\right), \label{A5B33443244}
\end{alignat}

\begin{alignat}{1}
\left(B_{3,3}^{4,4}\right)_{4,2}\cdot\left(B_{3,3}^{4,4}\right)_{3,4}
&=-y^{-1}\frac{\left[1\right]'\left[4\right]'}{\left[5\right]'\left[6\right]'}\left(\frac{\left[1\right]'\left[3\right]'\left(\left[3\right]'+\left[5\right]'\right)}{\left[5\right]'\left[4\right]'^2}-\frac{\left[1\right]'\left[3\right]'\left(\left[4\right]'+\left[6\right]'\right)}{\left[5\right]'^2\left[6\right]'} \right) \nonumber \\
&=-y^{-1}, \label{A5B33444234}
\end{alignat}

From  (\ref{A5B33443244}) and (\ref{A5B33444234}), we have

\begin{alignat}{1}
\left(B_{3,3}^{4,4}\right)_{4,4}\cdot\left(B_{3,3}^{4,4}\right)_{2,3}-\left(B_{3,3}^{4,4}\right)_{4,3}\cdot\left(B_{3,3}^{4,4}\right)_{2,4}
&=\sqrt{2}y^{-1}=1+i\not=0. \label{A5nonzero2}
\end{alignat}

\end{proof}

Then we consider braiding matrix for $L\left(\frac{21}{22},0\right)$-modules.
$ U^{1}=L\left(\frac{21}{22},0\right)$, $ U^{2}=L\left(\frac{21}{22},8\right)$ are
irreducible $L\left(\frac{21}{22},0\right)$-modules. For convenience,
we will denote $\left(B_{U^{a},U^{b}}^{U^{c},U^{d}}\right)_{U^{e},U^{f}}$
by $\left(B_{a,b}^{c,d}\right)_{e,f}$, $a,b,c$, $d,$ $e$,
$f$ $\in\left\{1,2\right\} $. Now we are ready to give the following
lemma.

\begin{lemma} \label{3Cnonzero}$\left(B_{2,2}^{2,2}\right)_{2,1}\not=0$. \end{lemma}
\begin{proof}
By a careful computation similar to Lemma\ref{5Anonzero}, we can obtain

\begin{alignat*}{1}
\left(B_{2,2}^{2,2}\right)_{2,1}=y^6\cdot\frac{\left[1\right]'^3}{\left[2\right]'^3}\cdot\frac{\left[1\right]'+\left[3\right]'}{\left[3\right]'}
&\cdot \frac{\left[10\right]'\left[9\right]'\left[8\right]'}{\left[7\right]'\left[6\right]'\left[5\right]'}\cdot\frac{\left[3\right]'\left[4\right]'+\left[1\right]'\left[4\right]'+\left[1\right]'\left[2\right]'}{\left[3\right]'\left[4\right]'}\not=0,
\end{alignat*}
where $\left[l\right]'=2i\sin\left(\frac{11}{12}\pi l\right)$, $y=\exp\left(\frac{11}{6}\pi i\right)$.

 \end{proof}

\section{Uniqueness of VOA structure of the $5A$-algebra $\mathcal{U}$}

As in \cite{LYY}, we denote the irreducible module $L\left(\frac{1}{2},h_{1}\right)\otimes L\left(\frac{25}{28},h_{2}\right)\otimes L\left(\frac{25}{28},h_{3}\right)$ by $[h_{1}, h_{2}, h_{3}]$ for simplicity of notation. Let

\begin{eqnarray*}
U^{1}=[0,0,0],\ U^{2}=[0, \frac{15}{2}, \frac{15}{2}],\ U^{3}=[0, \frac{3}{4}, \frac{13}{4}],\ U^{4}=[0, \frac{13}{4}, \frac{13}{4}],\\ U^{5}=[\frac{1}{2}, 0, \frac{15}{2}],\ U^{6}=[\frac{1}{2}, \frac{15}{2}, 0],
\ U^{7}=[\frac{1}{2}, \frac{3}{4}, \frac{3}{4}], \ U^{8}=[\frac{1}{2}, \frac{13}{4}, \frac{13}{4}], \\ U^{9}=[\frac{1}{16}, \frac{5}{32}, \frac{57}{32}], \ U^{10}=[\frac{1}{16}, \frac{57}{32}, \frac{5}{32}], \ U^{11}=[\frac{1}{16}, \frac{57}{32}, \frac{165}{32}], \ U^{12}=[\frac{1}{16}, \frac{165}{32}, \frac{57}{32}].
\end{eqnarray*}
Then

\[
\mathcal{U}\cong U^{1}\oplus U^{2}\oplus ...\oplus U^{11}\oplus U^{12}.
\]
\begin{lemma}\label{U1234}
Let $W=U^{1}+U^{2}+U^{3}+U^{4}$, then $W$ is a subVOA of $\mathcal{U}$ and the VOA structure of $6A$-algebra $\mathcal{U}$ is uniquely determined by $W$.
\end{lemma}
\begin{proof}
By Remark \ref{product property of qdim} and some results in \cite{DJX}, The quantum dimensions of $U^{i}$, $i=1,2,...,11,12$ as $U^{1}$ modules are as follows:

\[
q\dim_{U^{1}}U^{1}=q\dim_{U^{1}}U^{2}=q\dim_{U^{1}}U^{5}=q\dim_{U^{1}}U^{6}=1,
\]
\[
 q\dim_{U^{1}}U^{3}=q\dim_{U^{1}}U^{4}=q\dim_{U^{1}}U^{7}=q\dim_{U^{1}}U^{8}=\left(\frac{\sin(\frac{3\pi}{8})}{\sin(\frac{\pi}{8})}\right)^2,
 \]
 \[
 q\dim_{U^{1}}U^{9}=q\dim_{U^{1}}U^{10}=q\dim_{U^{1}}U^{11}=q\dim_{U^{1}}U^{12}=\frac{1}{\sin^2(\frac{\pi}{8})}.
\]

By the quantum dimensions above and the fusion rules of $L\left(\frac{1}{2},0\right)$ modules and $L\left(\frac{25}{28},0\right)$ modules, we can get that $U^{1}+U^{2}+...+U^{7}+U^{8}$ is a subVOA of $\mathcal{U}$, $U^{9}+U^{10}+U^{11}+U^{12}$ is a simple module of $U^{1}+U^{2}+...+U^{7}+U^{8}$ and $q\dim_{U^{1}+...+U^{8}}U^{9}+...+U^{12}=1$. By proposition \ref{qdim of simple current},
$U^{9}+U^{10}+U^{11}+U^{12}$ is a simple current module of $U^{1}+U^{2}+...+U^{7}+U^{8}$. Simialrly we can show that $U^{1}+U^{2}+U^{3}+U^{4}$ is a subVOA of $U^{1}+U^{2}+...+U^{7}+U^{8}$, $U^{5}+U^{6}+U^{7}+U^{8}$ is a simple current module of $U^{1}+U^{2}+U^{3}+U^{4}$. Hence $W$ is a subVOA of $\mathcal{U}$ and by Remark \ref{simple current extension}, the VOA structure of $6A$-algebra $\mathcal{U}$ is uniquely determined by $W$.
\end{proof}

 \begin{remark}\label{5Auniqueremark}
 Lemma \ref{U1234} tells us in order to prove uniqueness of VOA structure on $\mathcal{U}$, we only need to show the uniqueness of VOA structure on $W=U^{1}+U^{2}+U^{3}+U^{4}$. On the other hand, $W=L\left(\frac{1}{2},0\right)\otimes U$, here $U=L\left(\frac{25}{28},0\right)\otimes L\left(\frac{25}{28},0\right)+L\left(\frac{25}{28},\frac{15}{2}\right)\otimes L\left(\frac{25}{28},\frac{15}{2}\right)+L\left(\frac{25}{28},\frac{3}{4}\right)\otimes L\left(\frac{25}{28},\frac{13}{4}\right)+L\left(\frac{25}{28},\frac{13}{4}\right)\otimes L\left(\frac{25}{28},\frac{3}{4}\right)$. Now we only need to prove uniqueness of VOA structure on $U$. \end{remark}

\begin{remark} \label{self-dual}Since $U_{1}=0$ and $\dim U_{0}=1$
by Theorem \ref{bilinear form}, there is a unique bilinear form on
$U$ and thus $U'\cong U$. Without loss
of generality, we can identify $U$ with $U'$.\end{remark}

For convinience, we still use $U^{1}, U^{2}, U^{3}, U^{4}$ to denote $L\left(\frac{25}{28},0\right)\otimes L\left(\frac{25}{28},0\right), L\left(\frac{25}{28},\frac{15}{2}\right)\otimes L\left(\frac{25}{28},\frac{15}{2}\right), L\left(\frac{25}{28},\frac{3}{4}\right)\otimes L\left(\frac{25}{28},\frac{13}{4}\right), L\left(\frac{25}{28},\frac{13}{4}\right)\otimes L\left(\frac{25}{28},\frac{3}{4}\right)$ respectively.\\

we denote
\begin{gather}
P_{1}=Q_{1}=L\left(\frac{25}{28},0\right), \ P_{2}=Q_{2}=L\left(\frac{25}{28},\frac{15}{2}\right), \nonumber\\
P_{3}=Q_{4}=L\left(\frac{25}{28},\frac{3}{4}\right), \ P_{4}=Q_{3}=L\left(\frac{25}{28},\frac{13}{4}\right) \nonumber
\end{gather}
 and $U^{i}=P_{i}\otimes Q_{i}$, $i=1,2,3,4$. Then
\[
U\cong P_{1}\otimes Q_{1}\oplus P_{2}\otimes Q_{2}\oplus P_{3}\otimes Q_{3}\oplus P_{4}\otimes Q_{4}=U^{1}\oplus U^{2}\oplus U^{3}\oplus U^{4}.
\]

For convenience, we list the following tables of fusion rules

\begin{center}
\begin{tabular}{|c|c|c|c|}
\hline
$P_{1}$  & $P_{2}$  & $P_{3}$ & $P_{4}$\tabularnewline
\hline
\hline
$P_{2}$  & $P_{1}$  & $P_{4}$ & $P_{3}$\tabularnewline
\hline
$P_{3}$  & $P_{4}$  & $P_{1}+P_{3}+P_{4}$  & $P_{2}+P_{3}+P_{4}$\tabularnewline
\hline
$P_{4}$  & $P_{3}$  & $P_{2}+P_{3}+P_{4}$  & $P_{1}+P_{3}+P_{4}$\tabularnewline
\hline
\end{tabular}\label{Fusion rules for P_i's}
\par\end{center}

\begin{center}
\begin{tabular}{|c|c|c|c|}
\hline
$U^{1}$  & $U^{2}$  & $U^{3}$ & $U^{4}$\tabularnewline
\hline
\hline
$U^{2}$  & $U^{1}$  & $U^{4}$ & $U^{3}$\tabularnewline
\hline
$U^{3}$  & $U^{4}$  & $U^{1}+U^{3}+U^{4}$  & $U^{2}+U^{3}+U^{4}$\tabularnewline
\hline
$U^{4}$  & $U^{3}$  & $U^{2}+U^{3}+U^{4}$  & $U^{1}+U^{3}+U^{4}$\tabularnewline
\hline
\end{tabular}\label{Fusion rules for U^i's}
\par\end{center}

Let $\left(U,Y\right)$ be a vertex operator algebra structure
on $U$ with

\[
Y\left(u,z\right)=\sum_{a,b,c\in\left\{ 1,2,3,4\right\} }\text{\ensuremath{\lambda}}_{a,b}^{c}\cdot\mathcal{I}_{a,b}^{c}\left(u^{a},z\right)u^{b}
\]
where $\mathcal{I}_{a,b}^{c}$, $a,b,c\in\left\{ 1,2,3,4\right\} $ is a basis of $ I_{U^{1}}\left(_{U^{a}\ U^{b}}^{U^{c}}\right)$ . Furthermore, for each $u^{a}\in U^{a}$
we write $u^{a}=u_{1}^{a}\otimes u_{2}^{a}$ where $u_{1}^{a}\in P_{a}$
and $u_{2}^{a}\in Q_{a}$, $\mathcal{I}_{a,b}^{c}=\mathcal{Y}_{a,b}^{c}\otimes\overline{\mathcal{Y}}_{a,b}^{c}$
where $\mathcal{Y}_{a,b}^{c}\in I_{P_{1}}\left(_{P_{a}\ P_{b}}^{P_{c}}\right)$,
$\overline{\mathcal{Y}}_{a,b}^{c}\in I_{Q_{1}}\left(_{Q_{a}\ Q_{b}}^{Q_{c}}\right)$
with $a,b,c\in\left\{ 1,2,3,4\right\} $.

\begin{lemma} \label{A5nonzerolemma}$\lambda_{a,b}^{c}\not=0$ if $N_{a,b}^{c}=\dim I_{U^{1}}\left(_{U^{a}\ U^{b}}^{U^{c}}\right)\not=0$.\label{lemmaunique}
 \end{lemma}

\begin{claim}
$\lambda_{k,1}^{k}\not=0,\forall k=2,3,4$.
\end{claim}
\begin{proof}
For any $u^{k}\in U^{k}$, $k=1,2,3,4$, using skew symmetry of $Y\left(\cdot,z\right)$
(\cite{FHL}), we have

\[
Y(u^{k},z)u^{1}=e^{zL\left(-1\right)}Y\left(u^{1},-z\right)u^{k}=\lambda_{1,k}^{k}\cdot e^{zL\left(-1\right)}\mathcal{I}_{1,k}^{k}\left(u^{1},-z\right)u^{k}=\lambda_{k,1}^{k}\cdot\mathcal{I}_{k,1}^{k}\left(u^{k},z\right)u^{1}.
\]
Since $U^{k}$ is an irreducible $U^{1}$-module, we have $\lambda_{1,k}^{k}\not=0$,
$\forall k=1,2,3,4$. So $\lambda_{k,1}^{k}\not=0$, $\forall k=2,3,4.$
\end{proof}

\begin{claim}
$\lambda_{k,k}^{1}\not=0$, $\forall k=2,3,4$.
\end{claim}
\begin{proof}
By Remark \ref{self-dual}, $U$ has a unique invariant bilinear form $\left\langle \cdot,\cdot\right\rangle $
with $\left\langle 1,1\right\rangle =1$. For $u^{k},v^{k}\in U^{k},$
$k=1,2,3,4$, we have
\[
\left\langle Y\left(u^{k},z)v^{k}\right),u^{1}\right\rangle =\left\langle v^{k},Y\left(e^{zL\left(-1\right)}\left(-z^{-2}\right)^{L\left(0\right)}u^{k},z^{-1}\right)u^{1}\right\rangle .
\]
That is,
\[
\left\langle \lambda_{k,k}^{1}\cdot\mathcal{I}_{k,k}^{1}\left(u^{k},z\right)v^{k},u^{1}\right\rangle =\left\langle v^{k},\lambda_{k,1}^{k}\cdot\mathcal{I}_{k,1}^{k}\left(e^{zL\left(-1\right)}\left(-z^{-2}\right)^{L\left(0\right)}u^{k},z^{-1}\right)u^{1}\right\rangle .
\]
Applying previous claim, $\lambda_{k,1}^{k}\not=0$, hence $\lambda_{k,k}^{1}\not=0$,
$\forall k=2,3,4.$
\end{proof}

\begin{claim}
\label{lambda_2,3,2; 2,3,3; 3,3,2; 2,2,3; nonzero}\emph{ $\lambda_{3,2}^{4}$,
$\lambda_{2,3}^{4}$, $\lambda_{2,4}^{3},\lambda_{4,2}^{3}$, $\lambda_{3,4}^{2},\lambda_{4,3}^{2}$ are all nonzero .
}
\end{claim}

\begin{proof}
Let $u^{2}\in U^{2}$, $u^{3}\in U^{3}$, $u^{4}\in U^{4}$ . Skew symmetry of $Y$
gives
\[
\left\langle Y\left(u^{2},z\right)u^{3},u^{4}\right\rangle =\left\langle e^{zL\left(-1\right)}Y\left(u^{3},-z\right)u^{2},u^{4}\right\rangle ,
\]
that is,
\begin{equation}
\left\langle \lambda_{2,3}^{4}\cdot\mathcal{I}_{2,3}^{4}\left(u^{2},z\right)u^{3},u^{4}\right\rangle =\left\langle \lambda_{3,2}^{4}\cdot e^{zL\left(-1\right)}\mathcal{I}_{3,2}^{4}\left(u^{3},-z\right)u^{2},u^{4}\right\rangle .\label{A5nonzeroskew}
\end{equation}
So $\lambda_{2,3}^{4}$ and $\lambda_{3,2}^{4}$ are both zero or
nonzero. Similarly, we can get  $\lambda_{2,4}^{3}$ and $\lambda_{4,2}^{3}$ are both zero or
nonzero,  $\lambda_{4,3}^{2}$ and $\lambda_{3,4}^{2}$ are both zero or
nonzero.

For any $u^{1}\in U^{1},$ $u^{2},v^{2}\in U^{2}$, $u^{3}, v^{3}\in U^{3}$ and $u^{4}\in U^{4}$,
commutativity of $Y$ implies

\[
\iota_{12}^{-1}\left\langle u^{1},Y\left(u^{2},z_{1}\right)Y\left(u^{3},z_{2}\right)u^{4}\right\rangle =\iota_{21}^{-1}\left\langle u^{1},Y\left(u^{3},z_{1}\right)Y\left(u^{2},z_{1}\right)u^{4}\right\rangle ,
\]

\[
\iota_{12}^{-1}\left\langle u^{1},Y\left(u^{4},z_{1}\right)Y\left(u^{3},z_{2}\right)u^{2}\right\rangle =\iota_{21}^{-1}\left\langle u^{1},Y\left(u^{3},z_{1}\right)Y\left(u^{4},z_{1}\right)u^{2}\right\rangle ,
\]

\[
\iota_{12}^{-1}\left\langle v^{2},Y\left(u^{2},z_{1}\right)Y\left(u^{3},z_{2}\right)v^{3}\right\rangle =\iota_{21}^{-1}\left\langle v^{2},Y\left(u^{3},z_{1}\right)Y\left(u^{2},z_{1}\right)v^{3}\right\rangle .
\]

That is,

\begin{alignat}{1}
\iota_{12}^{-1} & \left\langle u^{1},\lambda_{2,2}^{1}\lambda_{3,4}^{2}\cdot\mathcal{I}_{2,2}^{1}\left(u^{2},z_{1}\right)\mathcal{I}_{3,4}^{2}\left(u^{3},z_{2}\right)u^{4}\right\rangle \nonumber \\
 & =\iota_{21}^{-1}\left\langle u^{1},\lambda_{3,3}^{1}\lambda_{2,4}^{3}\cdot\mathcal{I}_{3,3}^{1}\left(u^{3},z_{2}\right)\mathcal{I}_{2,4}^{3}\left(u^{2},z_{1}\right)u^{4}\right\rangle ,\label{A5nonzerocom2341}
\end{alignat}

\begin{alignat}{1}
\iota_{12}^{-1} & \left\langle u^{1},\lambda_{4,4}^{1}\lambda_{3,2}^{4}\cdot\mathcal{I}_{4,4}^{1}\left(u^{4},z_{1}\right)\mathcal{I}_{3,2}^{4}\left(u^{3},z_{2}\right)u^{2}\right\rangle \nonumber \\
 & =\iota_{21}^{-1}\left\langle u^{1},\lambda_{3,3}^{1}\lambda_{4,2}^{3}\cdot\mathcal{I}_{3,3}^{1}\left(u^{3},z_{2}\right)\mathcal{I}_{4,2}^{3}\left(u^{4},z_{1}\right)u^{2}\right\rangle ,\label{A5nonzerocom2342}
\end{alignat}

\begin{alignat}{1}
\iota_{12}^{-1} & \left\langle v^{2},\lambda_{2,1}^{2}\lambda_{3,3}^{1}\cdot\mathcal{I}_{2,1}^{2}\left(u^{2},z_{1}\right)\mathcal{I}_{3,3}^{1}\left(u^{3},z_{2}\right)v^{3}\right\rangle \nonumber \\
 & =\iota_{21}^{-1}\left\langle v^{2},\lambda_{3,4}^{2}\lambda_{2,3}^{4}\cdot\mathcal{I}_{3,4}^{2}\left(u^{3},z_{2}\right)\mathcal{I}_{2,3}^{4}\left(u^{2},z_{1}\right)v^{3}\right\rangle \label{A5nonzerocom2343}.
\end{alignat}

Then the Claim follows from (\ref{A5nonzeroskew}), (\ref{A5nonzerocom2341}), (\ref{A5nonzerocom2342}), (\ref{A5nonzerocom2343})  and  Claim 1 and 2.
\end{proof}

\begin{claim}
\emph{ $\lambda_{3,4}^{3}$,
$\lambda_{4,3}^{3}$, $\lambda_{3,3}^{4},\lambda_{4,4}^{4}, \lambda_{3,4}^{3}, \lambda_{3,4}^{4}$,
$\lambda_{4,3}^{4}$, $\lambda_{3,3}^{3},\lambda_{4,4}^{3}$ are all nonzero .
}
\end{claim}

\begin{proof}
First we will show that $\lambda_{4,3}^{3}$, $\lambda_{3,3}^{4},\lambda_{4,4}^{4}, \lambda_{3,4}^{3}$ are all zero or all nonzreo.

Let $u^{1} \in U^{1}, u^{2}\in U^{2}$, $v^{3}, u^{3}\in U^{3}$, $v^{4}, u^{4}\in U^{4}$ . Skew symmetry of $Y$
gives
\[
\left\langle Y\left(u^{3},z\right)u^{4},v^{3}\right\rangle =\left\langle e^{zL\left(-1\right)}Y\left(u^{4},-z\right)u^{3},v^{3}\right\rangle ,
\]
that is,
\begin{equation}
\left\langle \lambda_{3,4}^{3}\cdot\mathcal{I}_{3,4}^{3}\left(u^{3},z\right)u^{4},v^{3}\right\rangle =\left\langle \lambda_{4,3}^{3}\cdot e^{zL\left(-1\right)}\mathcal{I}_{4,3}^{3}\left(u^{4},-z\right)u^{3},v^{3}\right\rangle . \label{A53341}
\end{equation}
So $\lambda_{3,4}^{3}$ and $\lambda_{4,3}^{3}$ are both zero or
nonzero.

commutativity of $Y$ implies

\[
\iota_{12}^{-1}\left\langle u^{1},Y\left(u^{3},z_{1}\right)Y\left(u^{4},z_{2}\right)v^{3}\right\rangle =\iota_{21}^{-1}\left\langle u^{1},Y\left(u^{4},z_{1}\right)Y\left(u^{3},z_{1}\right)v^{3}\right\rangle ,
\]

\[
\iota_{12}^{-1}\left\langle u^{2},Y\left(u^{3},z_{1}\right)Y\left(u^{4},z_{2}\right)v^{4}\right\rangle =\iota_{21}^{-1}\left\langle u^{2},Y\left(u^{4},z_{1}\right)Y\left(u^{3},z_{1}\right)v^{4}\right\rangle ,
\]

That is,

\begin{alignat}{1}
\iota_{12}^{-1} & \left\langle u^{1},\lambda_{3,3}^{1}\lambda_{4,3}^{3}\cdot\mathcal{I}_{3,3}^{1}\left(u^{3},z_{1}\right)\mathcal{I}_{4,3}^{3}\left(u^{4},z_{2}\right)v^{3}\right\rangle \nonumber \\
 & =\iota_{21}^{-1}\left\langle u^{1},\lambda_{4,4}^{1}\lambda_{3,3}^{4}\cdot\mathcal{I}_{4,4}^{1}\left(u^{4},z_{2}\right)\mathcal{I}_{3,3}^{4}\left(u^{3},z_{1}\right)v^{3}\right\rangle ,\label{A53342}
\end{alignat}

\begin{alignat}{1}
\iota_{12}^{-1} & \left\langle u^{2},\lambda_{3,4}^{2}\lambda_{4,4}^{4}\cdot\mathcal{I}_{3,4}^{2}\left(u^{3},z_{1}\right)\mathcal{I}_{4,4}^{4}\left(u^{4},z_{2}\right)v^{4}\right\rangle \nonumber \\
 & =\iota_{21}^{-1}\left\langle u^{2},\lambda_{4,3}^{2}\lambda_{3,4}^{3}\cdot\mathcal{I}_{4,3}^{2}\left(u^{4},z_{2}\right)\mathcal{I}_{3,4}^{3}\left(u^{3},z_{1}\right)v^{4}\right\rangle ,\label{A53343}
\end{alignat}

Combining with Claim 1-3 and (\ref{A53341}), (\ref{A53342}), (\ref{A53343}) we have  $\lambda_{4,3}^{3}$ and $\lambda_{3,3}^{4}$ are both zero or
nonzero,  $\lambda_{4,4}^{4}$ and $\lambda_{3,4}^{3}$ are both zero or
nonzero. In total, $\lambda_{4,3}^{3}$, $\lambda_{3,3}^{4},\lambda_{4,4}^{4}, \lambda_{3,4}^{3}$ are all zero or all nonzreo.

Similarly, we can show that $\lambda_{3,4}^{4}$, $\lambda_{4,3}^{4},\lambda_{3,3}^{3}, \lambda_{4,4}^{3}$ are all zero or all nonzreo.

Next we use braiding matrices to finish this claim.

 Consider the four point functions on $\left(U^{3},U^{4},U^{4},U^{3}\right).$
Let $p_{1}^{3}\otimes p_{2}^{3}, t_{1}^{3}\otimes t_{2}^{3} \in U^{3}$, $u_{1}^{4}\otimes u_{2}^{4}, v_{1}^{4}\otimes v_{2}^{4}\in U^{4}$, let $B_{3,3}^{4,4}$ be as defined in (\ref{braiding matrix property}),
then we have
\begin{alignat}{1}
 & \iota_{12}^{-1}\left\langle t_{1}^{3}\otimes t_{2}^{3},Y\left(v_{1}^{4}\otimes v_{2}^{4},z_{1}\right)Y\left(u_{1}^{4}\otimes u_{2}^{4},z_{2}\right)p_{1}^{3}\otimes p_{2}^{3}\right\rangle \nonumber \\
 & =\iota_{12}^{-1}\langle t_{1}^{3}\otimes t_{2}^{3},\lambda_{4,2}^{3}\lambda_{4,3}^{2}\cdot\mathcal{Y}_{4,2}^{3}\otimes\overline{\mathcal{Y}}_{4,2}^{3}\left(v_{1}^{4}\otimes v_{2}^{4},z_{1}\right)\cdot\mathcal{Y}_{4,3}^{2}\otimes\overline{\mathcal{Y}}_{4,3}^{2}\left(u_{1}^{4}\otimes u_{2}^{4},z_{2}\right)\cdot p_{1}^{3}\otimes p_{2}^{3}\nonumber \\
 & +\lambda_{4,3}^{3}\lambda_{4,3}^{3}\cdot\mathcal{Y}_{4,3}^{3}\otimes\overline{\mathcal{Y}}_{4,3}^{3}\left(v_{1}^{4}\otimes v_{2}^{4},z_{1}\right)\cdot\mathcal{Y}_{4,3}^{3}\otimes\overline{\mathcal{Y}}_{4,3}^{3}\left(u_{1}^{4}\otimes u_{2}^{4},z_{2}\right)\cdot p_{1}^{3}\otimes p_{2}^{3}\rangle\nonumber \\
 & +\lambda_{4,4}^{3}\lambda_{4,3}^{4}\cdot\mathcal{Y}_{4,4}^{3}\otimes\overline{\mathcal{Y}}_{4,4}^{3}\left(v_{1}^{4}\otimes v_{2}^{4},z_{1}\right)\cdot\mathcal{Y}_{4,3}^{4}\otimes\overline{\mathcal{Y}}_{4,3}^{4}\left(u_{1}^{4}\otimes u_{2}^{4},z_{2}\right)\cdot p_{1}^{3}\otimes p_{2}^{3}\rangle\nonumber \\
 & =\iota_{12}^{-1}\langle t_{1}^{3}\otimes t_{2}^{3},\lambda_{4,2}^{3}\lambda_{4,3}^{2}\cdot\mathcal{Y}_{4,2}^{3}\left(v_{1}^{4},z_{1}\right)\mathcal{Y}_{4,3}^{2}\left(u_{1}^{4},z_{2}\right)p_{1}^{3}\otimes\overline{\mathcal{Y}}_{4,2}^{3}\left(v_{2}^{4},z_{1}\right)\overline{\mathcal{Y}}_{4,3}^{2}\left(u_{2}^{4},z_{2}\right)p_{2}^{3}\nonumber \\
 & +\lambda_{4,3}^{3}\lambda_{4,3}^{3}\cdot\mathcal{Y}_{4,3}^{3}\left(v_{1}^{4},z_{1}\right)\mathcal{Y}_{4,3}^{3}\left(u_{1}^{4},z_{2}\right)p_{1}^{3}\otimes\overline{\mathcal{Y}}_{4,3}^{3}\left(v_{2}^{4},z_{1}\right)\overline{\mathcal{Y}}_{4,3}^{3}\left(u_{2}^{4},z_{2}\right)p_{2}^{3}\nonumber \\
& +\lambda_{4,4}^{3}\lambda_{4,3}^{4}\cdot\mathcal{Y}_{4,4}^{3}\left(v_{1}^{4},z_{1}\right)\mathcal{Y}_{4,3}^{4}\left(u_{1}^{4},z_{2}\right)p_{1}^{3}\otimes\overline{\mathcal{Y}}_{4,4}^{3}\left(v_{2}^{4},z_{1}\right)\overline{\mathcal{Y}}_{4,3}^{4}\left(u_{2}^{4},z_{2}\right)p_{2}^{3}\rangle\nonumber \\
 & =\iota_{21}^{-1}\langle t_{1}^{3}\otimes t_{2}^{3},\lambda_{4,2}^{3}\lambda_{4,3}^{2}\cdot\sum_{i=2,3,4}\left(B_{3,3}^{4,4}\right)_{2,i}\mathcal{Y}_{4,i}^{3}\left(u_{1}^{4},z_{2}\right)\mathcal{Y}_{4,3}^{i}\left(v_{1}^{4},z_{1}\right)p_{1}^{3}\nonumber \\
 & \otimes\sum_{j=2,3,4}\left(\tilde{B}_{3,3}^{4,4}\right)_{2,j}\mathcal{\overline{Y}}_{4,j}^{3}\left(u_{2}^{4},z_{2}\right)\mathcal{\overline{Y}}_{4,3}^{j}\left(v_{2}^{4},z_{1}\right)p_{2}^{3}\nonumber \\
 &+ \lambda_{4,3}^{3}\lambda_{4,3}^{3}\cdot\sum_{i=2,3,4}\left(B_{3,3}^{4,4}\right)_{3,i}\mathcal{Y}_{4,i}^{3}\left(u_{1}^{4},z_{2}\right)\mathcal{Y}_{4,3}^{i}\left(v_{1}^{4},z_{1}\right)p_{1}^{3}\nonumber \\
 & \otimes\sum_{j=2,3,4}\left(\tilde{B}_{3,3}^{4,4}\right)_{3,j}\mathcal{\overline{Y}}_{4,j}^{3}\left(u_{2}^{4},z_{2}\right)\mathcal{\overline{Y}}_{4,3}^{j}\left(v_{2}^{4},z_{1}\right)p_{2}^{3}\nonumber\\
&+ \lambda_{4,4}^{3}\lambda_{4,3}^{4}\cdot\sum_{i=2,3,4}\left(B_{3,3}^{4,4}\right)_{4,i}\mathcal{Y}_{4,i}^{3}\left(u_{1}^{4},z_{2}\right)\mathcal{Y}_{4,3}^{i}\left(v_{1}^{4},z_{1}\right)p_{1}^{3}\nonumber \\
 & \otimes\sum_{j=2,3,4}\left(\tilde{B}_{3,3}^{4,4}\right)_{4,j}\mathcal{\overline{Y}}_{4,j}^{3}\left(u_{2}^{4},z_{2}\right)\mathcal{\overline{Y}}_{4,3}^{j}\left(v_{2}^{4},z_{1}\right)p_{2}^{3}\nonumber \rangle
\end{alignat}

In the mean time, we have
\begin{alignat}{1}
 & \iota_{21}^{-1}\left\langle t_{1}^{3}\otimes t_{2}^{3},Y\left(u_{1}^{4}\otimes u_{2}^{4},z_{2}\right)Y\left(v_{1}^{4}\otimes v_{2}^{4},z_{1}\right)p_{1}^{3}\otimes p_{2}^{3}\right\rangle \nonumber \\
 & =\iota_{21}^{-1}\langle t_{1}^{3}\otimes t_{2}^{3},\lambda_{4,2}^{3}\lambda_{4,3}^{2}\cdot\mathcal{Y}_{4,2}^{3}\left(u_{1}^{4},z_{2}\right)\mathcal{Y}_{4,3}^{2}\left(v_{1}^{4},z_{1}\right)p_{1}^{3}\otimes\mathcal{\overline{Y}}_{4,2}^{3}\left(u_{2}^{4},z_{2}\right)\mathcal{\overline{Y}}_{4,3}^{2}\left(v_{2}^{4},z_{1}\right)p_{2}^{3}\nonumber \\
 & +\lambda_{4,3}^{3}\lambda_{4,3}^{3}\cdot\mathcal{Y}_{4,3}^{3}\left(u_{1}^{4},z_{2}\right)\mathcal{Y}_{4,3}^{3}\left(v_{1}^{4},z_{1}\right)p_{1}^{3}\otimes\mathcal{\overline{Y}}_{4,3}^{3}\left(u_{2}^{4},z_{2}\right)\mathcal{\overline{Y}}_{4,3}^{3}\left(v_{2}^{4},z_{1}\right)p_{2}^{3}\nonumber \\
&+\lambda_{4,4}^{3}\lambda_{4,3}^{4}\cdot\mathcal{Y}_{4,4}^{3}\left(u_{1}^{4},z_{2}\right)\mathcal{Y}_{4,3}^{4}\left(v_{1}^{4},z_{1}\right)p_{1}^{3}\otimes\mathcal{\overline{Y}}_{4,4}^{3}\left(u_{2}^{4},z_{2}\right)\mathcal{\overline{Y}}_{4,3}^{4}\left(v_{2}^{4},z_{1}\right)p_{2}^{3}\nonumber \rangle
\end{alignat}

Then we can imply that
\begin{alignat}{1}
\begin{cases}
\lambda_{4,2}^{3}\lambda_{4,3}^{2}\left(B_{3,3}^{4,4}\right)_{2,2}\cdot\left(\tilde{B}_{3,3}^{4,4}\right)_{2,2}+\lambda_{4,3}^{3}\lambda_{4,3}^{3}\cdot\left(B_{3,3}^{4,4}\right)_{3,2}\cdot\left(\tilde{B}_{3,3}^{4,4}\right)_{3,2}+\lambda_{4,4}^{3}\lambda_{4,3}^{4}\cdot\left(B_{3,3}^{4,4}\right)_{4,2}\cdot\left(\tilde{B}_{3,3}^{4,4}\right)_{4,2}=\lambda_{4,2}^{3}\lambda_{4,3}^{2}\\
\lambda_{4,2}^{3}\lambda_{4,3}^{2}\left(B_{3,3}^{4,4}\right)_{2,3}\cdot\left(\tilde{B}_{3,3}^{4,4}\right)_{2,3}+\lambda_{4,3}^{3}\lambda_{4,3}^{3}\cdot\left(B_{3,3}^{4,4}\right)_{3,3}\cdot\left(\tilde{B}_{3,3}^{4,4}\right)_{3,3}+\lambda_{4,4}^{3}\lambda_{4,3}^{4}\cdot\left(B_{3,3}^{4,4}\right)_{4,3}\cdot\left(\tilde{B}_{3,3}^{4,4}\right)_{4,3}=\lambda_{4,3}^{3}\lambda_{4,3}^{3}\\
\lambda_{4,2}^{3}\lambda_{4,3}^{2}\left(B_{3,3}^{4,4}\right)_{2,4}\cdot\left(\tilde{B}_{3,3}^{4,4}\right)_{2,4}+\lambda_{4,3}^{3}\lambda_{4,3}^{3}\cdot\left(B_{3,3}^{4,4}\right)_{3,4}\cdot\left(\tilde{B}_{3,3}^{4,4}\right)_{3,4}+\lambda_{4,4}^{3}\lambda_{4,3}^{4}\cdot\left(B_{3,3}^{4,4}\right)_{4,4}\cdot\left(\tilde{B}_{3,3}^{4,4}\right)_{4,4}=\lambda_{4,4}^{3}\lambda_{4,3}^{4}\\
\lambda_{4,2}^{3}\lambda_{4,3}^{2}\left(B_{3,3}^{4,4}\right)_{2,2}\cdot\left(\tilde{B}_{3,3}^{4,4}\right)_{2,3}+\lambda_{4,3}^{3}\lambda_{4,3}^{3}\cdot\left(B_{3,3}^{4,4}\right)_{3,2}\cdot\left(\tilde{B}_{3,3}^{4,4}\right)_{3,3}+\lambda_{4,4}^{3}\lambda_{4,3}^{4}\cdot\left(B_{3,3}^{4,4}\right)_{4,2}\cdot\left(\tilde{B}_{3,3}^{4,4}\right)_{4,3}=0\\
\lambda_{4,2}^{3}\lambda_{4,3}^{2}\left(B_{3,3}^{4,4}\right)_{2,2}\cdot\left(\tilde{B}_{3,3}^{4,4}\right)_{2,4}+\lambda_{4,3}^{3}\lambda_{4,3}^{3}\cdot\left(B_{3,3}^{4,4}\right)_{3,2}\cdot\left(\tilde{B}_{3,3}^{4,4}\right)_{3,4}+\lambda_{4,4}^{3}\lambda_{4,3}^{4}\cdot\left(B_{3,3}^{4,4}\right)_{4,2}\cdot\left(\tilde{B}_{3,3}^{4,4}\right)_{4,4}=0. \label{A5nonzerosystem}\\
\lambda_{4,2}^{3}\lambda_{4,3}^{2}\left(B_{3,3}^{4,4}\right)_{2,3}\cdot\left(\tilde{B}_{3,3}^{4,4}\right)_{2,2}+\lambda_{4,3}^{3}\lambda_{4,3}^{3}\cdot\left(B_{3,3}^{4,4}\right)_{3,3}\cdot\left(\tilde{B}_{3,3}^{4,4}\right)_{3,2}+\lambda_{4,4}^{3}\lambda_{4,3}^{4}\cdot\left(B_{3,3}^{4,4}\right)_{4,3}\cdot\left(\tilde{B}_{3,3}^{4,4}\right)_{4,2}=0\\
\lambda_{4,2}^{3}\lambda_{4,3}^{2}\left(B_{3,3}^{4,4}\right)_{2,3}\cdot\left(\tilde{B}_{3,3}^{4,4}\right)_{2,4}+\lambda_{4,3}^{3}\lambda_{4,3}^{3}\cdot\left(B_{3,3}^{4,4}\right)_{3,3}\cdot\left(\tilde{B}_{3,3}^{4,4}\right)_{3,4}+\lambda_{4,4}^{3}\lambda_{4,3}^{4}\cdot\left(B_{3,3}^{4,4}\right)_{4,3}\cdot\left(\tilde{B}_{3,3}^{4,4}\right)_{4,4}=0\\
\lambda_{4,2}^{3}\lambda_{4,3}^{2}\left(B_{3,3}^{4,4}\right)_{2,4}\cdot\left(\tilde{B}_{3,3}^{4,4}\right)_{2,2}+\lambda_{4,3}^{3}\lambda_{4,3}^{3}\cdot\left(B_{3,3}^{4,4}\right)_{3,4}\cdot\left(\tilde{B}_{3,3}^{4,4}\right)_{3,2}+\lambda_{4,4}^{3}\lambda_{4,3}^{4}\cdot\left(B_{3,3}^{4,4}\right)_{4,4}\cdot\left(\tilde{B}_{3,3}^{4,4}\right)_{4,2}=0\\
\lambda_{4,2}^{3}\lambda_{4,3}^{2}\left(B_{3,3}^{4,4}\right)_{2,4}\cdot\left(\tilde{B}_{3,3}^{4,4}\right)_{2,3}+\lambda_{4,3}^{3}\lambda_{4,3}^{3}\cdot\left(B_{3,3}^{4,4}\right)_{3,4}\cdot\left(\tilde{B}_{3,3}^{4,4}\right)_{3,3}+\lambda_{4,4}^{3}\lambda_{4,3}^{4}\cdot\left(B_{3,3}^{4,4}\right)_{4,4}\cdot\left(\tilde{B}_{3,3}^{4,4}\right)_{4,3}=0 \\
\end{cases}
\end{alignat}

If  $\lambda_{4,3}^{3}=0$ , then from the first, sixth, eighth equations of (\ref{A5nonzerosystem})
we have
\[
\begin{cases}
\lambda_{4,2}^{3}\lambda_{4,3}^{2}\left(B_{3,3}^{4,4}\right)_{2,2}\cdot\left(\tilde{B}_{3,3}^{4,4}\right)_{2,2}+\lambda_{4,4}^{3}\lambda_{4,3}^{4}\cdot\left(B_{3,3}^{4,4}\right)_{4,2}\cdot\left(\tilde{B}_{3,3}^{4,4}\right)_{4,2}=\lambda_{4,2}^{3}\lambda_{4,3}^{2}\\
\lambda_{4,2}^{3}\lambda_{4,3}^{2}\left(B_{3,3}^{4,4}\right)_{2,3}\cdot\left(\tilde{B}_{3,3}^{4,4}\right)_{2,2}+\lambda_{4,4}^{3}\lambda_{4,3}^{4}\cdot\left(B_{3,3}^{4,4}\right)_{4,3}\cdot\left(\tilde{B}_{3,3}^{4,4}\right)_{4,2}=0\\
\lambda_{4,2}^{3}\lambda_{4,3}^{2}\left(B_{3,3}^{4,4}\right)_{2,4}\cdot\left(\tilde{B}_{3,3}^{4,4}\right)_{2,2}+\lambda_{4,4}^{3}\lambda_{4,3}^{4}\cdot\left(B_{3,3}^{4,4}\right)_{4,4}\cdot\left(\tilde{B}_{3,3}^{4,4}\right)_{4,2}=0\\
\end{cases}
\]

then we can get the following:
\begin{alignat}{1}
\begin{cases}
\lambda_{4,2}^{3}\lambda_{4,3}^{2}\left(B_{3,3}^{4,4}\right)_{2,2}\cdot\left(\tilde{B}_{3,3}^{4,4}\right)_{2,2}+\lambda_{4,4}^{3}\lambda_{4,3}^{4}\cdot\left(B_{3,3}^{4,4}\right)_{4,2}\cdot\left(\tilde{B}_{3,3}^{4,4}\right)_{4,2}=\lambda_{4,2}^{3}\lambda_{4,3}^{2} \label{A5systemnonzero1}\\
\lambda_{4,2}^{3}\lambda_{4,3}^{2}\left(\left(B_{3,3}^{4,4}\right)_{4,4}\cdot\left(B_{3,3}^{4,4}\right)_{2,3}-\left(B_{3,3}^{4,4}\right)_{4,3}\cdot\left(B_{3,3}^{4,4}\right)_{2,4}\right)\cdot\left(\tilde{B}_{3,3}^{4,4}\right)_{2,2}=0\\
\lambda_{4,4}^{3}\lambda_{4,3}^{4}\left(\left(B_{3,3}^{4,4}\right)_{4,4}\cdot\left(B_{3,3}^{4,4}\right)_{2,3}-\left(B_{3,3}^{4,4}\right)_{4,3}\cdot\left(B_{3,3}^{4,4}\right)_{2,4}\right)\cdot\left(\tilde{B}_{3,3}^{4,4}\right)_{4,2}=0\\
\end{cases}
\end{alignat}
Since by Lemma \ref{5Anonzero}
\[
\left(B_{3,3}^{4,4}\right)_{4,4}\cdot\left(B_{3,3}^{4,4}\right)_{2,3}-\left(B_{3,3}^{4,4}\right)_{4,3}\cdot\left(B_{3,3}^{4,4}\right)_{2,4}\not=0,\\
\]
 from the last two equations of (\ref{A5systemnonzero1}) we have
\[
\left(\tilde{B}_{3,3}^{4,4}\right)_{2,2}=0, \lambda_{4,4}^{3}\lambda_{4,3}^{4}\left(\tilde{B}_{3,3}^{4,4}\right)_{4,2}=0.\\
\]
But then by the first equation of (\ref{A5systemnonzero1}) we will get  $0=\lambda_{4,2}^{3}\lambda_{4,3}^{2}$, contradicting with Claim 3.
So  $\lambda_{4,3}^{3}\not=0$.\\

By a similar process, we can get $\lambda_{4,4}^{3}\not=0$. Hence \emph{ $\lambda_{3,4}^{3}$,
$\lambda_{4,3}^{3}$, $\lambda_{3,3}^{4},\lambda_{4,4}^{4}, \lambda_{3,4}^{3}, \lambda_{3,4}^{4}$,
$\lambda_{4,3}^{4}$, $\lambda_{3,3}^{3},\lambda_{4,4}^{3}$ are all nonzero .
}
\end{proof}

The following lemma was given in \cite{DJY}:
\begin{lemma}\label{new structure}Let $(V,Y)$ be a vertex operator
algebra and $\sigma:V\to V$ be a linear isomorphism such that $\sigma\left(1\right)=1,\sigma\left(\omega\right)=\omega$.
Then $(V,Y^{\sigma})$ is a vertex operator algebra where
\[
Y^{\sigma}(u,z)=\sigma Y(\sigma^{-1}u,z)\sigma^{-1}
\]
and $(V,Y)\cong(V,Y^{\sigma})$. \end{lemma}

Let $\left(U,\ Y\right)$ be a vertex operator algebra structure
on $U$. First we fix a basis $\left\{ \overline{\mathcal{Y}}_{a,b}^{c}\left(\cdot,z\right)|a,b,c=1,2,3,4\right\} $
for space of intertwining operators of type $\left(\begin{array}{c}
Q_{c}\\
Q_{a}\ Q_{b}
\end{array}\right)$,
$a,b,c\in\left\{ 1,2,3,4\right\} $ as in \cite{FFK}. By lemma \ref{A5nonzerolemma}, without loss
of generality, we can choose a basis $\left\{ \mathcal{Y}\left(\cdot,z\right)|a,b,c=1,2,3,4\right\} $
for space of intertwining operators of type $\left(\begin{array}{c}
P_{c}\\
P_{a}\ P_{b}
\end{array}\right)$,
$a,b,c\in\left\{ 1,2,3,4\right\} $ such that the coefficients $\lambda_{a,b}^{c}=1$
if $N_{a,b}^{c}\not=0$. Now we have $\left(U,Y\right)$,
a vertex operator algebra structure on $U=U^{1}\oplus U^{2}\oplus U^{3}\oplus U^{4}$
such that for any $u^{k},v^{k}\in U^{k}$, $k=1,2,3,4$,

\begin{gather}
Y\left(u^{k},z\right)u^{1}=\mathcal{I}_{k,1}^{k}\left(u^{k},z\right)u^{1}, k\in \left\{2,3,4\right\};\nonumber \\
Y\left(u^{2},z\right)u^{a}=\mathcal{I}_{2,a}^{b}\left(u^{2},z\right)u^{a}, \left\{a,b\right\}= \left\{3,4\right\};\nonumber \\
Y\left(u^{a},z\right)u^{2}=\mathcal{I}_{a,2}^{b}\left(u^{a},z\right)u^{2}, \left\{a,b\right\}= \left\{3,4\right\};\nonumber \\
Y\left(u^{2},z\right)v^{2}=\mathcal{I}_{2,2}^{1}\left(u^{2},z\right)v^{2};\nonumber \\
Y\left(u^{k},z\right)v^{k}=\mathcal{I}_{k,k}^{1}\left(u^{k},z\right)v^{k}+\mathcal{I}_{k,k}^{3}\left(u^{k},z\right)v^{k}+\mathcal{I}_{k,k}^{4}\left(u^{k},z\right)v^{k}, k\in \left\{3,4\right\};\nonumber \\
Y\left(u^{a},z\right)u^{b}=\mathcal{I}_{a,b}^{2}\left(u^{a},z\right)u^{b}+\mathcal{I}_{a,b}^{3}\left(u^{a},z\right)u^{b}+\mathcal{I}_{a,b}^{4}\left(u^{a},z\right)u^{b},  \left\{a,b\right\}= \left\{3,4\right\},
\label{U,Y}
\end{gather}

where $\mathcal{I}_{a,b}^{c}\in I_{U^{1}}\left(_{U^{a}\ U^{b}}^{U^{c}}\right)$,
$a,b,c\in\left\{ 1,2,3,4\right\} $ are nonzero intertwining operators.
Furthermore, for each $u^{i}\in U^{i}$, we write $u^{i}=u_{1}^{i}\otimes u_{2}^{i}$
where $u_{1}^{i}\in P_{i}$ and $u_{2}^{i}\in Q_{i}$, $\mathcal{I}_{a,b}^{c}=\mathcal{Y}_{a,b}^{c}\otimes\overline{\mathcal{Y}}_{a,b}^{c}$
where $\mathcal{Y}_{a,b}^{c}\in I_{P_{1}}\left(_{P_{a}\ P_{b}}^{P_{c}}\right)$,
$\overline{\mathcal{Y}}_{a,b}^{c}\in I_{Q_{1}}\left(_{Q_{a}\ Q_{b}}^{Q_{c}}\right)$
with $a,b,c\in\left\{ 1,2,3,4\right\} $.

\begin{theorem}\label{5Auniquethm}The vertex operator algebra structure on $U$
over $\mathbb{C}$ is unique.\end{theorem}
\begin{proof}
Let $\left(U,Y\right)$ be the vertex operator algebra structure
as given in (\ref{U,Y}). Suppose $\left(U,\overline{Y}\right)$
is another vertex operator algebra structure on $U$. Without
loss of generality, we may assume $Y\left(u,z\right)=\overline{Y}\left(u,z\right)$
for all $u\in U^{1}$. From our settings above, there exist nonzero
constants $\lambda_{i,1}^{i}$, $\lambda_{2,2}^{1}$, $\lambda_{2,j}^{k}$, $\lambda_{j,2}^{k}$
$\lambda_{3,4}^{p}$,$\lambda_{4,3}^{p}$, $\lambda_{33}^{l}$, $\lambda_{4,4}^{l}$where $i,p=2,3,4$, $\left\{j,k\right\}= \left\{3,4\right\}$,
$l=1,3,4$ such that for any $u^{i},v^{i}\in U^{i}$, $i=1,2,3,4$, we have
\[
\overline{Y}\left(u^{k},z\right)u^{1}=\lambda_{k,1}^{k}\cdot\mathcal{I}_{k,1}^{k}(u^{k},z)u^{1}, k\in \left\{2,3,4\right\};
\]
\[
\overline{Y}\left(u^{2},z\right)v^{2}=\lambda_{2,2}^{1}\cdot\mathcal{I}_{2,2}^{1}(u^{2},z)v^{2};
\]
\[
\overline{Y}\left(u^{2},z\right)u^{a}=\lambda_{2,a}^{b}\cdot\mathcal{I}_{2,a}^{b}(u^{2},z)u^{a}, \left\{a,b\right\}= \left\{3,4\right\};
\]
\[
\overline{Y}\left(u^{a},z\right)u^{2}=\lambda_{a,2}^{b}\cdot\mathcal{I}_{a,2}^{b}(u^{a},z)u^{2}, \left\{a,b\right\}= \left\{3,4\right\};
\]
\[
\overline{Y}\left(u^{k},z\right)v^{k}=\lambda_{k,k}^{1}\cdot\mathcal{I}_{k,k}^{1}\left(u^{k},z\right)v^{k}+\lambda_{k,k}^{3}\cdot\mathcal{I}_{k,k}^{3}\left(u^{k},z\right)v^{k}+\lambda_{k,k}^{4}\cdot\mathcal{I}_{k,k}^{4}\left(u^{k},z\right)v^{k}, k\in \left\{3,4\right\};
\]
\[
\overline{Y}\left(u^{a},z\right)u^{b}=\lambda_{a,b}^{2}\cdot\mathcal{I}_{a,b}^{2}\left(u^{a},z\right)u^{b}+\lambda_{a,b}^{3}\cdot\mathcal{I}_{a,b}^{3}\left(u^{a},z\right)u^{b}+\lambda_{a,b}^{4}\cdot\mathcal{I}_{a,b}^{4}\left(u^{a},z\right)u^{b}, \left\{a,b\right\}= \left\{3,4\right\};
\]
where $\mathcal{I}_{a,b}^{c}\in I_{U^{1}}\left(_{U^{a}\ U^{b}}^{U^{c}}\right),$
$a,b,c\in\left\{ 1,2,3,4\right\} $ are nonzero intertwining operators.

\emph{Claim 1) $\lambda_{k,1}^{k}=1, k\in \left\{2,3,4\right\}. $}

For any $u^{1}\in U^{1}$, $u^{k}\in U^{k},  k\in \left\{2,3,4\right\}$, skew symmetry of $Y\left(\cdot,z\right)$
and $\overline{Y}\left(\cdot,z\right)$ ( \cite{FHL} ) imply

\[
\overline{Y}(u^{k},z)u^{1}=e^{zL\left(-1\right)}\overline{Y}\left(u^{1},-z\right)u^{k}=e^{zL\left(-1\right)}Y(u^{1},-z)u^{k}=Y\left(u^{k},z\right)u^{1}=\mathcal{I}_{k,1}^{k}\left(u^{k},z\right)u^{1}.
\]
In the mean time, $\overline{Y}\left(u^{k},z\right)u^{1}=\lambda_{k,1}^{k}\cdot\mathcal{I}_{k,1}^{k}(u^{k},z)u^{1}$.
Thus we get $\lambda_{k,1}^{k}=1$.

\emph{Claim 2) $\lambda_{k,k}^{1}=1, k\in \left\{2,3,4\right\}. $}

Note that by Remark \ref{self-dual}, $U$ has a unique
invariant bilinear form $\left\langle \cdot,\cdot\right\rangle $
with $\left\langle 1,1\right\rangle =1$. For $u^{1}\in U^{1}$ and
$u^{k},v^{k}\in U^{k},  k\in \left\{2,3,4\right\}$, we have
\[
\left\langle Y\left(u^{k},z)v^{k}\right),u^{1}\right\rangle =\left\langle v^{k},Y\left(e^{zL\left(-1\right)}\left(-z^{-2}\right)^{L\left(0\right)}u^{k},z^{-1}\right)u^{1}\right\rangle .
\]
That is,
\[
\left\langle \mathcal{I}_{k,k}^{1}\left(u^{k},z\right)v^{k},u^{1}\right\rangle =\left\langle v^{k},\mathcal{I}_{k,1}^{k}\left(e^{zL\left(-1\right)}\left(-z^{-2}\right)^{L\left(0\right)}u^{k},z^{-1}\right)u^{1}\right\rangle .
\]
The invariant bilinear form on $\left(U,\overline{Y}\right)$
gives
\[
\left\langle \lambda_{k,k}^{1}\cdot\mathcal{I}_{k,k}^{1}\left(u^{k},z\right)v^{k},u^{1}\right\rangle =\left\langle v^{k},\lambda_{k,1}^{k}\cdot\mathcal{I}_{k,1}^{k}\left(e^{zL\left(-1\right)}\left(-z^{-2}\right)^{L\left(0\right)}u^{k},z^{-1}\right)u^{1}\right\rangle .
\]
Using  claim 1, we get $\lambda_{k,k}^{1}=1$.

\emph{Claim 3) $\lambda_{2,3}^{4}=\lambda_{3,2}^{4}=\lambda_{2,4}^{3}=\lambda_{4,2}^{3}=\lambda_{4,3}^{2}=\lambda_{3,4}^{2}=\lambda$,
 $\lambda^2=1$. }

Let $u^{2}\in U^{2}$, $u^{3}\in U^{3}$, $u^4 \in U^{4}$. by skew symmetry of
$Y$ we obtain
\[
\left\langle Y\left(u^{2},z\right)u^{3},u^{4}\right\rangle =\left\langle e^{zL\left(-1\right)}Y\left(u^{3},-z\right)u^{2},u^{4}\right\rangle ,
\]
that is,
\[
\left\langle \mathcal{I}_{2,3}^{4}\left(u^{2},z\right)u^{3}u^{4}\right\rangle =\left\langle e^{zL\left(-1\right)}\mathcal{I}_{3,2}^{4}\left(u^{3},-z\right)u^{2},u^{4}\right\rangle .
\]
Skew symmetry of $\overline{Y}$ gives
\[
\lambda_{2,3}^{4}\left\langle \mathcal{I}_{2,3}^{4}\left(u^{2},z\right)u^{3},u^{4}\right\rangle =\lambda_{3,2}^{4}\left\langle e^{zL\left(-1\right)}\mathcal{I}_{3,2}^{4}\left(u^{3},-z\right)u^{2},u^{4}\right\rangle .
\]
The above two identities gives $\lambda_{2,3}^{4}=\lambda_{3,2}^{4}.$
Similarly, we can prove $\lambda_{2,4}^{3}=\lambda_{4,2}^{3}$, $\lambda_{3,4}^{2}=\lambda_{4,3}^{2}$.

For any $u^{1}\in U^{1},$ $u^{2},v^{2}\in U^{2}$, $u^{3}, v^{3}\in U^{3}$ and $u^{4}\in U^{4}$,
Community of $Y$ implies
\begin{alignat}{1}
\iota_{12}^{-1} & \left\langle u^{1},\mathcal{I}_{2,2}^{1}\left(u^{2},z_{1}\right)\mathcal{I}_{3,4}^{2}\left(u^{3},z_{2}\right)u^{4}\right\rangle \nonumber \\
 & =\iota_{21}^{-1}\left\langle u^{1},\mathcal{I}_{3,3}^{1}\left(u^{3},z_{2}\right)\mathcal{I}_{2,4}^{3}\left(u^{2},z_{1}\right)u^{4}\right\rangle ,  \nonumber
\end{alignat}

Community of $\overline{Y}$ implies
\begin{alignat}{1}
\iota_{12}^{-1} & \left\langle u^{1},\lambda_{2,2}^{1}\lambda_{3,4}^{2}\cdot\mathcal{I}_{2,2}^{1}\left(u^{2},z_{1}\right)\mathcal{I}_{3,4}^{2}\left(u^{3},z_{2}\right)u^{4}\right\rangle \nonumber \\
 & =\iota_{21}^{-1}\left\langle u^{1},\lambda_{3,3}^{1}\lambda_{2,4}^{3}\cdot\mathcal{I}_{3,3}^{1}\left(u^{3},z_{2}\right)\mathcal{I}_{2,4}^{3}\left(u^{2},z_{1}\right)u^{4}\right\rangle ,  \nonumber
\end{alignat}
The above two identities and claim 2) together give us
\begin{equation}
\lambda_{3,4}^{2}=\lambda_{2,4}^{3}.\label{unique201}
\end{equation}
Similarly, from (\ref{A5nonzerocom2342}) and (\ref{A5nonzerocom2343}), we can get
\begin{equation}
\lambda_{3,2}^{4}=\lambda_{4,2}^{3}\label{unique202},
\end{equation}
\begin{equation}
\lambda_{3,4}^{2}\cdot\lambda_{2,3}^{4}=1\label{unique203}.
\end{equation}
So $\lambda_{2,3}^{4}=\lambda_{3,2}^{4}=\lambda_{2,4}^{3}=\lambda_{4,2}^{3}=\lambda_{4,3}^{2}=\lambda_{3,4}^{2}=\lambda$,
 $\lambda^2=1$.

\emph{Claim 4) $\lambda_{3,4}^{3}=\lambda_{4,3}^{4}=\lambda_{3,3}^{4}=\lambda_{4,4}^{4}=\mu$, $\lambda_{3,4}^{4}=\lambda_{4,3}^{4}=\lambda_{4,4}^{3}=\lambda_{3,3}^{3}=\gamma$,
 $\mu^2=\gamma^2=1$. }

The proof of equalities $\lambda_{3,4}^{3}=\lambda_{4,3}^{4}=\lambda_{3,3}^{4}=\lambda_{4,4}^{4}$, $\lambda_{3,4}^{4}=\lambda_{4,3}^{4}=\lambda_{4,4}^{3}=\lambda_{3,3}^{3}$ is similar to Claim 3, we denote them as $\mu, \gamma$ respectively. Now we mainly focus on the proof
 $\mu^2=\gamma^2=1$.

Consider the four point functions on $\left(U^{3},U^{4},U^{4},U^{3}\right).$ For $\left(U,Y\right)$ and $\left(U,\overline{Y}\right)$ , a similar process as (\ref{A5nonzerosystem}) gives
\[
\begin{cases}
\left(B_{3,3}^{4,4}\right)_{2,2}\cdot\left(\tilde{B}_{3,3}^{4,4}\right)_{2,2}+\left(B_{3,3}^{4,4}\right)_{3,2}\cdot\left(\tilde{B}_{3,3}^{4,4}\right)_{3,2}+\left(B_{3,3}^{4,4}\right)_{4,2}\cdot\left(\tilde{B}_{3,3}^{4,4}\right)_{4,2}=1\\
\left(B_{3,3}^{4,4}\right)_{2,3}\cdot\left(\tilde{B}_{3,3}^{4,4}\right)_{2,3}+\left(B_{3,3}^{4,4}\right)_{3,3}\cdot\left(\tilde{B}_{3,3}^{4,4}\right)_{3,3}+\left(B_{3,3}^{4,4}\right)_{4,3}\cdot\left(\tilde{B}_{3,3}^{4,4}\right)_{4,3}=1\\
\left(B_{3,3}^{4,4}\right)_{2,4}\cdot\left(\tilde{B}_{3,3}^{4,4}\right)_{2,4}+\left(B_{3,3}^{4,4}\right)_{3,4}\cdot\left(\tilde{B}_{3,3}^{4,4}\right)_{3,4}+\left(B_{3,3}^{4,4}\right)_{4,4}\cdot\left(\tilde{B}_{3,3}^{4,4}\right)_{4,4}=1\\
\left(B_{3,3}^{4,4}\right)_{2,2}\cdot\left(\tilde{B}_{3,3}^{4,4}\right)_{2,3}+\left(B_{3,3}^{4,4}\right)_{3,2}\cdot\left(\tilde{B}_{3,3}^{4,4}\right)_{3,3}+\left(B_{3,3}^{4,4}\right)_{4,2}\cdot\left(\tilde{B}_{3,3}^{4,4}\right)_{4,3}=0\\
\left(B_{3,3}^{4,4}\right)_{2,2}\cdot\left(\tilde{B}_{3,3}^{4,4}\right)_{2,4}+\left(B_{3,3}^{4,4}\right)_{3,2}\cdot\left(\tilde{B}_{3,3}^{4,4}\right)_{3,4}+\left(B_{3,3}^{4,4}\right)_{4,2}\cdot\left(\tilde{B}_{3,3}^{4,4}\right)_{4,4}=0\\
\left(B_{3,3}^{4,4}\right)_{2,3}\cdot\left(\tilde{B}_{3,3}^{4,4}\right)_{2,2}+\left(B_{3,3}^{4,4}\right)_{3,3}\cdot\left(\tilde{B}_{3,3}^{4,4}\right)_{3,2}+\left(B_{3,3}^{4,4}\right)_{4,3}\cdot\left(\tilde{B}_{3,3}^{4,4}\right)_{4,2}=0\\
\left(B_{3,3}^{4,4}\right)_{2,3}\cdot\left(\tilde{B}_{3,3}^{4,4}\right)_{2,4}+\left(B_{3,3}^{4,4}\right)_{3,3}\cdot\left(\tilde{B}_{3,3}^{4,4}\right)_{3,4}+\left(B_{3,3}^{4,4}\right)_{4,3}\cdot\left(\tilde{B}_{3,3}^{4,4}\right)_{4,4}=0\\
\left(B_{3,3}^{4,4}\right)_{2,4}\cdot\left(\tilde{B}_{3,3}^{4,4}\right)_{2,2}+\left(B_{3,3}^{4,4}\right)_{3,4}\cdot\left(\tilde{B}_{3,3}^{4,4}\right)_{3,2}+\left(B_{3,3}^{4,4}\right)_{4,4}\cdot\left(\tilde{B}_{3,3}^{4,4}\right)_{4,2}=0\\
\left(B_{3,3}^{4,4}\right)_{2,4}\cdot\left(\tilde{B}_{3,3}^{4,4}\right)_{2,3}+\left(B_{3,3}^{4,4}\right)_{3,4}\cdot\left(\tilde{B}_{3,3}^{4,4}\right)_{3,3}+\left(B_{3,3}^{4,4}\right)_{4,4}\cdot\left(\tilde{B}_{3,3}^{4,4}\right)_{4,3}=0\\
\end{cases}
\]
\[
\begin{cases}
\left(B_{3,3}^{4,4}\right)_{2,2}\cdot\left(\tilde{B}_{3,3}^{4,4}\right)_{2,2}+\mu^2\cdot\left(B_{3,3}^{4,4}\right)_{3,2}\cdot\left(\tilde{B}_{3,3}^{4,4}\right)_{3,2}+\gamma^2\cdot\left(B_{3,3}^{4,4}\right)_{4,2}\cdot\left(\tilde{B}_{3,3}^{4,4}\right)_{4,2}=1\\
\left(B_{3,3}^{4,4}\right)_{2,3}\cdot\left(\tilde{B}_{3,3}^{4,4}\right)_{2,3}+\mu^2\cdot\left(B_{3,3}^{4,4}\right)_{3,3}\cdot\left(\tilde{B}_{3,3}^{4,4}\right)_{3,3}+\gamma^2\cdot\left(B_{3,3}^{4,4}\right)_{4,3}\cdot\left(\tilde{B}_{3,3}^{4,4}\right)_{4,3}=\mu^2\\
\left(B_{3,3}^{4,4}\right)_{2,4}\cdot\left(\tilde{B}_{3,3}^{4,4}\right)_{2,4}+\mu^2\cdot\left(B_{3,3}^{4,4}\right)_{3,4}\cdot\left(\tilde{B}_{3,3}^{4,4}\right)_{3,4}+\gamma^2\cdot\left(B_{3,3}^{4,4}\right)_{4,4}\cdot\left(\tilde{B}_{3,3}^{4,4}\right)_{4,4}=\gamma^2\\
\left(B_{3,3}^{4,4}\right)_{2,2}\cdot\left(\tilde{B}_{3,3}^{4,4}\right)_{2,3}+\mu^2\cdot\left(B_{3,3}^{4,4}\right)_{3,2}\cdot\left(\tilde{B}_{3,3}^{4,4}\right)_{3,3}+\gamma^2\cdot\left(B_{3,3}^{4,4}\right)_{4,2}\cdot\left(\tilde{B}_{3,3}^{4,4}\right)_{4,3}=0\\
\left(B_{3,3}^{4,4}\right)_{2,2}\cdot\left(\tilde{B}_{3,3}^{4,4}\right)_{2,4}+\mu^2\cdot\left(B_{3,3}^{4,4}\right)_{3,2}\cdot\left(\tilde{B}_{3,3}^{4,4}\right)_{3,4}+\gamma^2\cdot\left(B_{3,3}^{4,4}\right)_{4,2}\cdot\left(\tilde{B}_{3,3}^{4,4}\right)_{4,4}=0\\
\left(B_{3,3}^{4,4}\right)_{2,3}\cdot\left(\tilde{B}_{3,3}^{4,4}\right)_{2,2}+\mu^2\cdot\left(B_{3,3}^{4,4}\right)_{3,3}\cdot\left(\tilde{B}_{3,3}^{4,4}\right)_{3,2}+\gamma^2\cdot\left(B_{3,3}^{4,4}\right)_{4,3}\cdot\left(\tilde{B}_{3,3}^{4,4}\right)_{4,2}=0\\
\left(B_{3,3}^{4,4}\right)_{2,3}\cdot\left(\tilde{B}_{3,3}^{4,4}\right)_{2,4}+\mu^2\cdot\left(B_{3,3}^{4,4}\right)_{3,3}\cdot\left(\tilde{B}_{3,3}^{4,4}\right)_{3,4}+\gamma^2\cdot\left(B_{3,3}^{4,4}\right)_{4,3}\cdot\left(\tilde{B}_{3,3}^{4,4}\right)_{4,4}=0\\
\left(B_{3,3}^{4,4}\right)_{2,4}\cdot\left(\tilde{B}_{3,3}^{4,4}\right)_{2,2}+\mu^2\cdot\left(B_{3,3}^{4,4}\right)_{3,4}\cdot\left(\tilde{B}_{3,3}^{4,4}\right)_{3,2}+\gamma^2\cdot\left(B_{3,3}^{4,4}\right)_{4,4}\cdot\left(\tilde{B}_{3,3}^{4,4}\right)_{4,2}=0\\
\left(B_{3,3}^{4,4}\right)_{2,4}\cdot\left(\tilde{B}_{3,3}^{4,4}\right)_{2,3}+\mu^2\cdot\left(B_{3,3}^{4,4}\right)_{3,4}\cdot\left(\tilde{B}_{3,3}^{4,4}\right)_{3,3}+\gamma^2\cdot\left(B_{3,3}^{4,4}\right)_{4,4}\cdot\left(\tilde{B}_{3,3}^{4,4}\right)_{4,3}=0\\
\end{cases}
\]
the two systems above give

\[
\begin{cases}
(1-\mu^2)\cdot\left(B_{3,3}^{4,4}\right)_{3,2}\cdot\left(\tilde{B}_{3,3}^{4,4}\right)_{3,2}+(1-\gamma^2)\cdot\left(B_{3,3}^{4,4}\right)_{4,2}\cdot\left(\tilde{B}_{3,3}^{4,4}\right)_{4,2}=0\\
(1-\mu^2)\cdot\left(B_{3,3}^{4,4}\right)_{3,3}\cdot\left(\tilde{B}_{3,3}^{4,4}\right)_{3,3}+(1-\gamma^2)\cdot\left(B_{3,3}^{4,4}\right)_{4,3}\cdot\left(\tilde{B}_{3,3}^{4,4}\right)_{4,3}=1-\mu^2\\
(1-\mu^2)\cdot\left(B_{3,3}^{4,4}\right)_{3,4}\cdot\left(\tilde{B}_{3,3}^{4,4}\right)_{3,4}+(1-\gamma^2)\cdot\left(B_{3,3}^{4,4}\right)_{4,4}\cdot\left(\tilde{B}_{3,3}^{4,4}\right)_{4,4}=1-\gamma^2\\
(1-\mu^2)\cdot\left(B_{3,3}^{4,4}\right)_{3,2}\cdot\left(\tilde{B}_{3,3}^{4,4}\right)_{3,3}+(1-\gamma^2)\cdot\left(B_{3,3}^{4,4}\right)_{4,2}\cdot\left(\tilde{B}_{3,3}^{4,4}\right)_{4,3}=0\\
(1-\mu^2)\cdot\left(B_{3,3}^{4,4}\right)_{3,2}\cdot\left(\tilde{B}_{3,3}^{4,4}\right)_{3,4}+(1-\gamma^2)\cdot\left(B_{3,3}^{4,4}\right)_{4,2}\cdot\left(\tilde{B}_{3,3}^{4,4}\right)_{4,4}=0\\
(1-\mu^2)\cdot\left(B_{3,3}^{4,4}\right)_{3,3}\cdot\left(\tilde{B}_{3,3}^{4,4}\right)_{3,2}+(1-\gamma^2)\cdot\left(B_{3,3}^{4,4}\right)_{4,3}\cdot\left(\tilde{B}_{3,3}^{4,4}\right)_{4,2}=0\\
(1-\mu^2)\cdot\left(B_{3,3}^{4,4}\right)_{3,3}\cdot\left(\tilde{B}_{3,3}^{4,4}\right)_{3,4}+(1-\gamma^2)\cdot\left(B_{3,3}^{4,4}\right)_{4,3}\cdot\left(\tilde{B}_{3,3}^{4,4}\right)_{4,4}=0\\
(1-\mu^2)\cdot\left(B_{3,3}^{4,4}\right)_{3,4}\cdot\left(\tilde{B}_{3,3}^{4,4}\right)_{3,2}+(1-\gamma^2)\cdot\left(B_{3,3}^{4,4}\right)_{4,4}\cdot\left(\tilde{B}_{3,3}^{4,4}\right)_{4,2}=0\\
(1-\mu^2)\cdot\left(B_{3,3}^{4,4}\right)_{3,4}\cdot\left(\tilde{B}_{3,3}^{4,4}\right)_{3,3}+(1-\gamma^2)\cdot\left(B_{3,3}^{4,4}\right)_{4,4}\cdot\left(\tilde{B}_{3,3}^{4,4}\right)_{4,3}=0\\
\end{cases}
\]
If $1-\mu^2 \not=0$, by using the first, fourth, ninth equations of the system above, we have
\[
\begin{cases}
(1-\mu^2)\cdot\left(B_{3,3}^{4,4}\right)_{3,3}\cdot\left(\tilde{B}_{3,3}^{4,4}\right)_{3,3}+(1-\gamma^2)\cdot\left(B_{3,3}^{4,4}\right)_{4,3}\cdot\left(\tilde{B}_{3,3}^{4,4}\right)_{4,3}=1-\mu^2\\
(1-\mu^2)\cdot\left(B_{3,3}^{4,4}\right)_{3,2}\cdot\left(\tilde{B}_{3,3}^{4,4}\right)_{3,3}+(1-\gamma^2)\cdot\left(B_{3,3}^{4,4}\right)_{4,2}\cdot\left(\tilde{B}_{3,3}^{4,4}\right)_{4,3}=0\\
(1-\mu^2)\cdot\left(B_{3,3}^{4,4}\right)_{3,4}\cdot\left(\tilde{B}_{3,3}^{4,4}\right)_{3,3}+(1-\gamma^2)\cdot\left(B_{3,3}^{4,4}\right)_{4,4}\cdot\left(\tilde{B}_{3,3}^{4,4}\right)_{4,3}=0\\
\end{cases}
\]
Then by a simple computation we have
\begin{alignat}{1}
\begin{cases}
(1-\mu^2)\cdot\left(B_{3,3}^{4,4}\right)_{3,3}\cdot\left(\tilde{B}_{3,3}^{4,4}\right)_{3,3}+(1-\gamma^2)\cdot\left(B_{3,3}^{4,4}\right)_{4,3}\cdot\left(\tilde{B}_{3,3}^{4,4}\right)_{4,3}=1-\mu^2\\
(1-\mu^2)\cdot\left(\left(B_{3,3}^{4,4}\right)_{3,2}\cdot\left(B_{3,3}^{4,4}\right)_{4,4}-\left(B_{3,3}^{4,4}\right)_{4,2}\cdot\left(B_{3,3}^{4,4}\right)_{3,4}\right)\cdot\left(\tilde{B}_{3,3}^{4,4}\right)_{3,3}=0\\
(1-\gamma^2)\cdot\left(\left(B_{3,3}^{4,4}\right)_{3,2}\cdot\left(B_{3,3}^{4,4}\right)_{4,4}-\left(B_{3,3}^{4,4}\right)_{4,2}\cdot\left(B_{3,3}^{4,4}\right)_{3,4}\right)\cdot\left(\tilde{B}_{3,3}^{4,4}\right)_{4,3}=0 \label{5Athmsys1}\\
\end{cases}
\end{alignat}
By lemma \ref{5Anonzero},  $\left(B_{3,3}^{4,4}\right)_{3,2}\cdot\left(B_{3,3}^{4,4}\right)_{4,4}-\left(B_{3,3}^{4,4}\right)_{4,2}\cdot\left(B_{3,3}^{4,4}\right)_{3,4}\not=0$, so from the last two equations of (\ref{5Athmsys1}) we have $\left(\tilde{B}_{3,3}^{4,4}\right)_{3,3}=0$ and $(1-\gamma^2)\cdot\left(\tilde{B}_{3,3}^{4,4}\right)_{4,3}=0$, But then from
the first equation of (\ref{5Athmsys1}), we get $1-\mu^2=0$ which is a contradiction. So $\mu^2=1$. Similarly, we can show that
$\gamma^2=1$.

\emph{Claim 5) $(U, Y)$ is isomorphic to  $\left(U,\overline{Y}\right)$.}

Define a linear map $\sigma$ such that
\[
\sigma|_{U^{1}}=1,\ \sigma|_{U^{2}}=\lambda\mu\gamma,\ \sigma|_{U^{3}}=\gamma,\ \sigma|_{U^{4}}=\mu
\]
where $\lambda^2=\mu^2=\gamma^2=1$. It is clear that $\sigma$
is a linear isomorphism of $U$. Using Lemma \ref{new structure},
$\sigma$ gives a vertex operator algebra structure $(U,Y^{\sigma})$
with $Y^{\sigma}(u,z)=\sigma Y(\sigma^{-1}u,z)\sigma^{-1}$ which
is isomorphic to $(U,Y)$. It is easy to verify that $Y^{\sigma}\left(u,z\right)=\overline{Y}\left(u,z\right)$
for all $u\in U$. Thus we proved the uniqueness of the vertex
operator algebra structure on $U$ .
\end{proof}
\begin{theorem}
The vertex operator algebra structure on $5A$-algebra $\mathcal{U}$ over $\mathbb{C}$ is unique.
\end{theorem}
\begin{proof}
The theorem follows immediately from Remark \ref{5Auniqueremark}
and Theorem \ref{5Auniquethm}.
\end{proof}
\section{Uniqueness of VOA structure of the $3C$-algebra $\mathcal{U}$}

From here forward, we denote
\begin{eqnarray*}
U^{1}=L\left(\frac{1}{2},0\right)\otimes L\left(\frac{21}{22},0\right),\ U^{2}=L\left(\frac{1}{2},0\right)\otimes L\left(\frac{21}{22},8\right),\ U^{3}=L\left(\frac{1}{2},\frac{1}{2}\right)\otimes L\left(\frac{21}{22},\frac{45}{2}\right),\\ U^{4}=L\left(\frac{1}{2},\frac{1}{2}\right)\otimes L\left(\frac{21}{22},\frac{7}{2}\right),\ U^{5}=L\left(\frac{1}{2},\frac{1}{16}\right)\otimes L\left(\frac{21}{22},\frac{31}{16}\right),\ U^{6}=L\left(\frac{1}{2},\frac{1}{16}\right)\otimes L\left(\frac{21}{22},\frac{175}{16}\right).
\end{eqnarray*}
Then from \cite{LYY}, the $3C$-algebra
\[
\mathcal{U}\cong U^{1}\oplus U^{2}\oplus U^{3}\oplus U^{4}\oplus U^{5}\oplus U^{6}.
\]

 \begin{lemma}\label{U12}
Let $W=U^{1}+U^{2}$, then $W$ is a subVOA of $\mathcal{U}$ and the VOA structure of $6A$-algebra $\mathcal{U}$ is uniquely determined by $W$.
\end{lemma}
\begin{proof}
 The quantum dimensions of $U^{i}$, $i=1,2,3,4,5,6$ as $U_{1}$ modules are as follows:

 \[
q\dim_{U^{1}}U^{1}=q\dim_{U^{1}}U^{3}=1,
\]
\[
 q\dim_{U^{1}}U^{2}=q\dim_{U^{1}}U^{4}=\frac{\sin(\frac{5\pi}{12})}{\sin(\frac{\pi}{12})},
 \]
 \[
  q\dim_{U^{1}}U^{5}=q\dim_{U^{1}}U^{6}=\sqrt{2}\frac{\sin(\frac{4\pi}{12})}{\sin(\frac{\pi}{12})}.
\]

Combining with the fusion rules of $L\left(\frac{1}{2},0\right)$ modules and $L\left(\frac{21}{22},0\right)$ modules, we can get that $U_{1}+U_{2}+U_{3}+U_{4}$ is a subVOA of $\mathcal{U}$,
$U^{5}+U^{6}$ is a simple module of $U^{1}+U^{2}+U^{3}+U^{4}$ and $q\dim_{U^{1}+...+U^{4}}U^{5}+U^{6}=1$. By proposition \ref{qdim of simple current}, $U^{5}+U^{6}$ is a simple current module of $U^{1}+U^{2}+U^{3}+U^{4}$. Similarly we can get $W=U^{1}+U^{2}$ is a subVOA of $U^{1}+U^{2}+U^{3}+U^{4}$, $U^{3}+U^{4}$ is a simple current module of $W$. Hence $W$ is a subVOA of $\mathcal{U}$ and by Remark \ref{simple current extension}, the VOA structure of $3C$-algebra $\mathcal{U}$ is uniquely determined by $W$.
So in order to prove uniqueness of VOA structure on $\mathcal{U}$, we only need to show the uniqueness of VOA structure on $U^{1}+U^{2}$. Since $U^{1}+U^{2}=L\left(\frac{1}{2},0\right)
\otimes \left(L\left(\frac{21}{22},0\right)+L\left(\frac{21}{22},8\right)\right)$, so it is enough to show the uniqueness of VOA structure on $L\left(\frac{21}{22},0\right)+L\left(\frac{21}{22},8\right)$.
\end{proof}

\begin{remark}\label{3Cuniqueremark}
 Lemma \ref{U12} tells us in order to prove uniqueness of VOA structure on $\mathcal{U}$, we only need to show the uniqueness of VOA structure on $W=U^{1}+U^{2}$. On the other hand, $W=L\left(\frac{1}{2},0\right)\otimes U$, here $U=L\left(\frac{21}{22},0\right)+L\left(\frac{21}{22},8\right)$. Now we only need to prove uniqueness of VOA structure on $U$. \end{remark}

\begin{remark} \label{3Cself-dual}Since $U_{1}=0$ and $\dim U_{0}=1$
by Theorem \ref{bilinear form}, there is a unique bilinear form on
$U$ and thus $U'\cong U$. Without loss
of generality, we can identify $U$ with $U'$.\end{remark}

From now on, we let $U=U^{1}+U^{2}$, $U^{1}=L\left(\frac{21}{22},0\right)$, $U^{2}=L\left(\frac{21}{22},8\right)$.
Let $\left(U,Y\right)$ be a vertex operator algebra structure
on $U$ with

\[
Y\left(u,z\right)=\sum_{a,b,c\in\left\{ 1,2\right\} }\text{\ensuremath{\lambda}}_{a,b}^{c}\cdot\mathcal{I}_{a,b}^{c}\left(u^{a},z\right)u^{b}
\]
where $\mathcal{I}_{a,b}^{c}$, $a,b,c\in\left\{ 1,2\right\} $ is a basis of $ I_{U^{1}}\left(_{U^{a}\ U^{b}}^{U^{c}}\right)$ .
\begin{lemma} $\lambda_{a,b}^{c}\not=0$ if $N_{a,b}^{c}=\dim I_{U^{1}}\left(_{U^{a}\ U^{b}}^{U^{c}}\right)\not=0$.\label{lemmaunique3C}
 \end{lemma}
\begin{claim}
$\lambda_{2,1}^{2}\not=0$.
\end{claim}
\begin{proof}
For any $u^{k}\in U^{k}$, $k=1,2$, using skew symmetry of $Y\left(\cdot,z\right)$
(\cite{FHL}), we have

\[
Y(u^{2},z)u^{1}=e^{zL\left(-1\right)}Y\left(u^{1},-z\right)u^{2}=\lambda_{1,2}^{2}\cdot e^{zL\left(-1\right)}\mathcal{I}_{1,2}^{2}\left(u^{1},-z\right)u^{2}=\lambda_{2,1}^{2}\cdot\mathcal{I}_{2,1}^{2}\left(u^{2},z\right)u^{1}.
\]
Since $U^{2}$ is an irreducible $U^{1}$-module, we have $\lambda_{1,2}^{2}\not=0$,
. So $\lambda_{2,1}^{2}\not=0$.
\end{proof}

\begin{claim}
$\lambda_{2,2}^{1}\not=0$.
\end{claim}
\begin{proof}
By Remark \ref{3Cself-dual}, $U$ has a unique invariant bilinear form $\left\langle \cdot,\cdot\right\rangle $
with $\left\langle 1,1\right\rangle =1$. For $u^{k},v^{k}\in U^{k},$
$k=1,2$, we have
\[
\left\langle Y\left(u^{2},z)v^{2}\right),u^{1}\right\rangle =\left\langle v^{2},Y\left(e^{zL\left(-1\right)}\left(-z^{-2}\right)^{L\left(0\right)}u^{2},z^{-1}\right)u^{1}\right\rangle .
\]
That is,
\[
\left\langle \lambda_{2,2}^{1}\cdot\mathcal{I}_{2,2}^{1}\left(u^{2},z\right)v^{2},u^{1}\right\rangle =\left\langle v^{2},\lambda_{2,1}^{2}\cdot\mathcal{I}_{2,1}^{2}\left(e^{zL\left(-1\right)}\left(-z^{-2}\right)^{L\left(0\right)}u^{2},z^{-1}\right)u^{1}\right\rangle .
\]
Applying previous claim, $\lambda_{2,1}^{2}\not=0$, hence $\lambda_{2,2}^{1}\not=0$.
\end{proof}
\begin{claim}
$\lambda_{2,2}^{2}\not=0$.
\end{claim}

\begin{proof}
Assume $\lambda_{2,2}^{2}=0$. Then we have $U^{1}.U^{2}=U^{2}$, $U^{2}.U^{1}=U^{2},$ $U^{2}.U^{2}=U^{1}$.
Define $\sigma:U^{1}+U^{2}\to U^{1}+U^{2}$ such that $\sigma|_{U^{1}}=1$
and $\sigma|_{U^{2}}=-1$. Then $\sigma$ is an order 2 automorphism
of $U^{1}+U^{2}$ with $\left(U^{1}+U^{2}\right)^{\sigma}=U^{1}$
and $U^{2}$ is a $U^{1}$-module.By Theorem \ref{classical galois theory} and Theorem \ref{quantum dimension and orbifold module}, $q\dim_{U^{1}}U^{2}=1$ because any irreducible representation of the group generated by $\sigma$ is 1-dimensional, contradicting with the fact that
$q\dim_{U^{1}}U^{2}=\frac{\sin(\frac{5\pi}{12})}{\sin(\frac{\pi}{12})}\not=0$.
Therefore, $\lambda_{2,2}^{2}\not=0$.
\end{proof}
Let $\left(U,\ Y\right)$ be a vertex operator algebra structure
on $U$. Without loss
of generality, we can choose a basis $\mathcal{I}_{a,b}^{c}\in I_{U^{1}}\left(_{U^{a}\ U^{b}}^{U^{c}}\right)$, $a,b,c\in\left\{1,2\right\} $ such that the coefficients $\lambda_{a,b}^{c}=1$
if $N_{a,b}^{c}\not=0$. Now we have $\left(U,\ Y\right)$,
a vertex operator algebra structure on $U=U^{1}\oplus U^{2}$
such that for any $u^{k},v^{k}\in U^{k}$, $k=1,2$,

\begin{gather}
Y\left(u^{2},z\right)u^{1}=\mathcal{I}_{2,1}^{2}\left(u^{2},z\right)u^{1};\nonumber \\
Y\left(u^{2},z\right)v^{2}=\mathcal{I}_{2,2}^{2}\left(u^{2},z\right)v^{2}; \nonumber \\
Y\left(u^{2},z\right)v^{2}=\mathcal{I}_{2,2}^{1}\left(u^{2},z\right)v^{2}. \nonumber \\
\label{C3Y}
\end{gather}
\begin{theorem}\label{3Cuniquethm}The vertex operator algebra structure on $U$
over $\mathbb{C}$ is unique.\end{theorem}
\begin{proof}
Let $\left(U,Y\right)$ be the vertex operator algebra structure
as given in (\ref{C3Y}). Suppose $\left(U,\overline{Y}\right)$
is another vertex operator algebra structure on $U$. Without
loss of generality, we may assume $Y\left(u,z\right)=\overline{Y}\left(u,z\right)$
for all $u\in U^{1}$. From our settings above, there exist nonzero
constants $\lambda_{2,1}^{2}$, $\lambda_{2,2}^{1}$, $\lambda_{2,2}^{2}$ such that for any $u^{i},v^{i}\in U^{i}$, $i=1,2$, we have
\[
\overline{Y}\left(u^{2},z\right)u^{1}=\lambda_{2,1}^{2}\cdot\mathcal{I}_{2,1}^{2}(u^{2},z)u^{1};
\]
\[
\overline{Y}\left(u^{2},z\right)v^{2}=\lambda_{2,2}^{1}\cdot\mathcal{I}_{2,2}^{1}(u^{2},z)v^{2};
\]
\[
\overline{Y}\left(u^{2},z\right)u^{2}=\lambda_{2,2}^{2}\cdot\mathcal{I}_{2,2}^{2}(u^{2},z)u^{2};
\]
where $\mathcal{I}_{a,b}^{c}\in I_{U^{1}}\left(_{U^{a}\ U^{b}}^{U^{c}}\right),$
$a,b,c\in\left\{1,2\right\} $ are nonzero intertwining operators.\\

\emph{Claim 1) $\lambda_{2,1}^{2}=1. $}

For any $u^{1}\in U^{1}$, $u^{2}\in U^{2}$, skew symmetry of $Y\left(\cdot,z\right)$
and $\overline{Y}\left(\cdot,z\right)$ ( \cite{FHL} ) imply

\[
\overline{Y}(u^{2},z)u^{1}=e^{zL\left(-1\right)}\overline{Y}\left(u^{1},-z\right)u^{2}=e^{zL\left(-1\right)}Y(u^{1},-z)u^{2}=Y\left(u^{2},z\right)u^{1}=\mathcal{I}_{2,1}^{2}\left(u^{2},z\right)u^{1}.
\]
In the mean time, $\overline{Y}\left(u^{2},z\right)u^{1}=\lambda_{2,1}^{2}\cdot\mathcal{I}_{2,1}^{2}(u^{2},z)u^{1}$.
Thus we get $\lambda_{2,1}^{2}=1$.

\emph{Claim 2) $\lambda_{2,2}^{1}=1. $}

Note that by Remark \ref{3Cself-dual}, $U$ has a unique
invariant bilinear form $\left\langle \cdot,\cdot\right\rangle $
with $\left\langle 1,1\right\rangle =1$. For $u^{1}\in U^{1}$ and
$u^{2},v^{2}\in U^{2},$ we have
\[
\left\langle Y\left(u^{2},z)v^{2}\right),u^{1}\right\rangle =\left\langle v^{2},Y\left(e^{zL\left(-1\right)}\left(-z^{-2}\right)^{L\left(0\right)}u^{2},z^{-1}\right)u^{1}\right\rangle .
\]
That is,
\[
\left\langle \mathcal{I}_{2,2}^{1}\left(u^{2},z\right)v^{2},u^{1}\right\rangle =\left\langle v^{2},\mathcal{I}_{2,1}^{2}\left(e^{zL\left(-1\right)}\left(-z^{-2}\right)^{L\left(0\right)}u^{2},z^{-1}\right)u^{1}\right\rangle .
\]
The invariant bilinear form on $\left(U,\overline{Y}\right)$
gives
\[
\left\langle \lambda_{2,2}^{1}\cdot\mathcal{I}_{2,2}^{1}\left(u^{2},z\right)v^{2},u^{1}\right\rangle =\left\langle v^{2},\lambda_{2,1}^{2}\cdot\mathcal{I}_{2,1}^{2}\left(e^{zL\left(-1\right)}\left(-z^{-2}\right)^{L\left(0\right)}u^{2},z^{-1}\right)u^{1}\right\rangle .
\]
Using  claim 1, we get $\lambda_{2,2}^{1}=1$.

\emph{Claim 3) $\lambda_{2,2}^{2}=\pm1. $}

For simplicity, we denote $\lambda_{2,2}^{2}:=\lambda$. Consider the four point functions on $\left(U^{2},U^{2},U^{2},U^{2}\right).$ For $\left(U,\overline{Y}\right)$
Let $p^{2}, t^{2}, u^{2}, v^{2}\in U^{2}$,
we have
\begin{alignat}{1}
 &  \iota_{12}^{-1}\left\langle t^{2},\overline{Y}\left(v^{2},z_{1}\right)\overline{Y}\left(u^{2},z_{2}\right)p^{2}\right\rangle \nonumber \\
 & = \iota_{12}^{-1}\langle t^{2},\mathcal{I}_{2,1}^{2}\left(v^{2},z_{1}\right)\mathcal{I}_{2,2}^{1}\left(u^{2},z_{2}\right)\cdot p^{2}+\lambda^2\mathcal{I}_{2,2}^{2}\left(v^{2},z_{1}\right)\mathcal{I}_{2,2}^{2}\left(u^{2},z_{2}\right)\cdot p^{2}
\rangle\nonumber \\
 & = \iota_{21}^{-1}\langle t^{2},\sum_{i=1,2}\left(B_{2,2}^{2,2}\right)_{1,i}\mathcal{I}_{2,i}^{2}\left(u^{2},z_{2}\right)\mathcal{I}_{2,2}^{i}\left(v^{2},z_{1}\right)\cdot p^{2}+\lambda^{2}\cdot\sum_{i=1,2}\left(B_{2,2}^{2,2}\right)_{2,i}\mathcal{I}_{2,i}^{2}\left(u^{2},z_{2}\right)\mathcal{I}_{2,2}^{i}\left(v^{2},z_{1}\right)\cdot p^{2}\nonumber \rangle
\end{alignat}

In the mean time,
\begin{alignat}{1}
 &  \iota_{21}^{-1}\left\langle t^{2},\overline{Y}\left(u^{2},z_{2}\right)\overline{Y}\left(v^{2},z_{1}\right)p^{2}\right\rangle \nonumber \\
 & = \iota_{21}^{-1}\langle t^{2},\mathcal{I}_{2,1}^{2}\left(u^{2},z_{2}\right)\mathcal{I}_{2,2}^{1}\left(v^{2},z_{1}\right)\cdot p^{2}+\lambda^2\mathcal{I}_{2,2}^{2}\left(u^{2},z_{2}\right)\mathcal{I}_{2,2}^{2}\left(v^{2},z_{1}\right)\cdot p^{2}
\nonumber \rangle
\end{alignat}
Then we can imply that
\[
\begin{cases}
\left(B_{2,2}^{2,2}\right)_{1,1}+\lambda^2\cdot\left(B_{2,2}^{2,2}\right)_{2,1}=1\\
\left(B_{2,2}^{2,2}\right)_{1,2}+\lambda^2\cdot\left(B_{2,2}^{2,2}\right)_{2,2}=\lambda^2\\
\end{cases}
\]
Similarly, for $\left(U,Y\right)$, we have
\[
\begin{cases}
\left(B_{2,2}^{2,2}\right)_{1,1}+\left(B_{2,2}^{2,2}\right)_{2,1}=1\\
\left(B_{2,2}^{2,2}\right)_{1,2}+\left(B_{2,2}^{2,2}\right)_{2,2}=1\\
\end{cases}
\]
From these two systems of equation, we can get
\[
\begin{cases}
(1-\lambda^2)\cdot\left(B_{2,2}^{2,2}\right)_{2,1}=0\\
(1-\lambda^2)\cdot\left(B_{2,2}^{2,2}\right)_{2,2}=1-\lambda^2\\
\end{cases}
\]
By Lemma \ref{3Cnonzero}, we have $\left(B_{2,2}^{2,2}\right)_{2,1}\not=0$, which implies $\lambda^2=1$.

\emph{Claim 4) $(U, Y)$ is isomorphic to  $\left(U,\overline{Y}\right)$.}

Define a linear map $\sigma$ such that
\[
\sigma|_{U^{1}}=1,\ \sigma|_{U^{2}}=\lambda
\]
where $\lambda^2=1$. It is clear that $\sigma$
is a linear isomorphism of $U$. Using Lemma \ref{new structure},
$\sigma$ gives a vertex operator algebra structure $(U,Y^{\sigma})$
with $Y^{\sigma}(u,z)=\sigma Y(\sigma^{-1}u,z)\sigma^{-1}$ which
is isomorphic to $(U,Y)$. It is easy to verify that $Y^{\sigma}\left(u,z\right)=\overline{Y}\left(u,z\right)$
for all $u\in U$. Thus we proved the uniqueness of the vertex
operator algebra structure on $U$ .
\end{proof}

\begin{theorem}
The vertex operator algebra structure on $3C$-algebra $\mathcal{U}$ over $\mathbb{C}$ is unique.
\end{theorem}
\begin{proof}
The theorem follows immediately from Remark \ref{3Cuniqueremark}
and Theorem \ref{3Cuniquethm}.
\end{proof}

\section{Fusion rules}
In this section, we will use the following result:

\begin{proposition} (\cite{ADL})\label{restriction of fusion rules} Let $V$
be a vertex operator algebra and let $W^{1}$, $W^{2}$, $W^{3}$
be $V$-modules among which $W^{1}$ and $W^{2}$ are irreducible.
Suppose that $V_{0}$ is a vertex operator subalgebra of $V$ (with
the same Virasoro element) and that $N^{1}$ and $N^{2}$ are irreducible
$V_{0}$-modules of $W^{1}$ and $W^{2}$, respectively. Then the
restriction map from $I_{V}\left(_{W^{1}\ W^{2}}^{\ \ W^{3}}\right)$
to $I_{V_{0}}\left(_{N^{1}\ N^{2}}^{\ \ W^{3}}\right)$ is injective.
In particular,
\[
\dim I_{V}\left(_{W^{1\ }W^{2}}^{\ \ W^{3}}\right)\le\dim I_{V_{0}}\left(_{N^{1}\ N^{2}}^{\ \ W^{3}}\right).
\]
\end{proposition}

\subsection{fusion rules of the $5A$-algebra $\mathcal{U}$}
First we need the following theorem:
\begin{theorem} (theorem 3.19 in \cite{LYY})\label{irrmod5A}
There are exactly nine irreducible modules $\mathcal{U}(i,j)$, $i,j=1, 3, 5$, for $\mathcal{U}$. As $L\left(\frac{1}{2},0\right)\otimes L\left(\frac{25}{28},0\right)\otimes L\left(\frac{25}{28},0\right)$-modules, they are of the following form:
\begin{eqnarray*}
\mathcal{U}(i,j)\cong&[0, h_{i,1}, h_{j,1}]\oplus [0, h_{i,3}, h_{j,5}]\oplus [0, h_{i,5}, h_{j,3}]\oplus [0, h_{i,7}, h_{j,7}]\\
&\oplus [\frac{1}{2}, h_{i,1}, h_{j,7}]\oplus [\frac{1}{2}, h_{i,3}, h_{j,3}]\oplus [\frac{1}{2}, h_{i,5}, h_{j,5}]\oplus [\frac{1}{2}, h_{i,7}, h_{j,1}]\\
&\oplus [\frac{1}{16}, h_{i,2}, h_{j,4}]\oplus [\frac{1}{16}, h_{i,4}, h_{j,2}]\oplus [\frac{1}{16}, h_{i,6}, h_{j,4}]\oplus [\frac{1}{16}, h_{i,4}, h_{j,6}],
\end{eqnarray*}
where $h_{m,n}=\frac{(7n-8m)^2-1}{4\cdot7\cdot8}$.
\end{theorem}
Now we can state our theorem:
\begin{theorem}
 $\dim I_{\mathcal{U}}\left(_{\mathcal{U}(i,j) \ \mathcal{U}(i',j')}^{\ \ \ \mathcal{U}(i'',j'')}\right)=1$ iff both $\left(\left(i,1\right),\left(i',1\right),\left(i'',1\right)\right)$ and $\left(\left(j,1\right),\left(j',1\right),\left(j'',1\right)\right)$
are admissible triples of pairs for $p=7, q=8$ (see definition \ref{admissiblepair})  and $0$ otherwise.
\end{theorem}

\begin{proof}
theorem \ref{restriction of fusion rules} implies the following inequality:
\[
\dim I_{\mathcal{U}}\left(_{\mathcal{U}(i,j) \ \mathcal{U}(i',j')}^{\ \ \ \mathcal{U}(i'',j'')}\right)\le \dim I_{[0, 0, 0]}\left(_{[0, h_{i,1}, h_{j,1}] \ [0, h_{i',1}, h_{j',1}]}^{\ \ \ \ \ \ \ \ \ \ \mathcal{U}(i'',j'')}\right)=\dim I_{[0, 0, 0]}\left(_{[0, h_{i,1}, h_{j,1}] \ [0, h_{i',1}, h_{j',1}]}^{\ \ \ \ \ \ \ \ [0, h_{i'',1}, h_{j'',1}]}\right).
\]
On the other hand, by directly computation, we have
\[
 q\dim_{\mathcal{U}}\mathcal{U}(i,j)=q\dim_{[0, 0, 0]}[0, h_{i,1}, h_{j,1}]= \frac{\sin(\frac{8i\pi}{7})}{\sin(\frac{8\pi}{7})}\cdot\frac{\sin(\frac{8j\pi}{7})}{\sin(\frac{8\pi}{7})}.
\]
So we have
\begin{eqnarray*}
&\dim I_{\mathcal{U}}\left(_{\mathcal{U}(i,j) \ \mathcal{U}(i',j')}^{\ \ \ \mathcal{U}(i'',j'')}\right)= \dim I_{[0, 0, 0]}\left(_{[0, h_{i,1}, h_{j,1}] \ [0, h_{i',1}, h_{j',1}]}^{\ \ \ \ \ \ \ \ \ \ \mathcal{U}(i'',j'')}\right)\\
&=\dim I_{L\left(\frac{25}{28},0\right)}\left(_{L\left(\frac{25}{28},h_{i,1}\right) \ L\left(\frac{25}{28},h_{i',1}\right)}^{\ \ \ \ \ \ \ \ L\left(\frac{25}{28},h_{i'',1}\right)}\right)\cdot\dim I_{L\left(\frac{25}{28},0\right)}\left(_{L\left(\frac{25}{28},h_{j,1}\right) \ L\left(\frac{25}{28},h_{j',1}\right)}^{\ \ \ \ \ \ \ \ L\left(\frac{25}{28},h_{j'',1}\right)}\right).
\end{eqnarray*}
Then we can conclude our theorem by using theorem \ref{fusion rules of virasoro modules}.
\end{proof}

\subsection{fusion rules of the $3C$-algebra $\mathcal{U}$}
First we need the following theorem:
\begin{theorem} (theorem 3.38 in \cite{LYY})\label{irrmod3C}
There are exactly five irreducible $\mathcal{U}$-modules $\mathcal{U}(2k)$, $0\le k\le 4$. In fact, $\mathcal{U}(0)=\mathcal{U}$ and as $L\left(\frac{1}{2},0\right)\otimes L\left(\frac{21}{22},0\right)$-modules,

\begin{eqnarray*}
\mathcal{U}(2)\cong[0, \frac{13}{11}]\oplus [0, \frac{35}{11}]\oplus [\frac{1}{2}, \frac{15}{22}]\oplus [\frac{1}{2}, \frac{301}{22}]\oplus [\frac{1}{16}, \frac{21}{176}]\oplus [\frac{1}{16}, \frac{901}{176}]\\
\mathcal{U}(4)\cong[0, \frac{6}{11}]\oplus [0, \frac{50}{11}]\oplus [\frac{1}{2}, \frac{1}{22}]\oplus [\frac{1}{2}, \frac{155}{22}]\oplus [\frac{1}{16}, \frac{85}{176}]\oplus [\frac{1}{16}, \frac{261}{176}]\\
\mathcal{U}(6)\cong[0, \frac{1}{11}]\oplus [0, \frac{111}{11}]\oplus [\frac{1}{2}, \frac{35}{22}]\oplus [\frac{1}{2}, \frac{57}{22}]\oplus [\frac{1}{16}, \frac{5}{176}]\oplus [\frac{1}{16}, \frac{533}{176}]\\
\mathcal{U}(8)\cong[0, \frac{20}{11}]\oplus [0, \frac{196}{11}]\oplus [\frac{1}{2}, \frac{7}{22}]\oplus [\frac{1}{2}, \frac{117}{22}]\oplus [\frac{1}{16}, \frac{133}{176}]\oplus [\frac{1}{16}, \frac{1365}{176}].
\end{eqnarray*}
\end{theorem}
Now we can state our theorem:
\begin{theorem}
 $\dim I_{\mathcal{U}(0)}\left(_{\mathcal{U}(i) \ \mathcal{U}(j)}^{\ \ \mathcal{U}(k)}\right)=1$ iff $\left(\left(i+1,1\right),\left(j+1,1\right),\left(k+1,1\right)\right)$
is an admissible triple of pairs for $p=11, q=12$ (see definition \ref{admissiblepair})  and $0$ otherwise.
\end{theorem}

\begin{proof}
Let $h_{m,n}=\frac{(11n-12m)^2-1}{4\cdot11\cdot12}$. Then for the irreducible $ L\left(\frac{21}{22},0\right)$-module $ L\left(\frac{21}{22}, h_{m,n}\right)$, $h_{m,n}=0, \frac{13}{11}, \frac{50}{11}, \frac{111}{11}, \frac{196}{11}$ correspond to $(m,n)=(1,1), (3,1), (5,1), (7,1), (9,1)$ respectively.
If we use the pair $(m,n)$ to denote the irreducible $L\left(\frac{1}{2},0\right)\otimes L\left(\frac{21}{22},0\right)$-module $[0, h_{m,n}]$, then by theorem
\ref{restriction of fusion rules}, we have
\[
\dim I_{\mathcal{U}(0)}\left(_{\mathcal{U}(i) \ \mathcal{U}(j)}^{\ \ \mathcal{U}(k)}\right)\le \dim I_{L\left(\frac{1}{2},0\right)\otimes L\left(\frac{21}{22},0\right)}\left(_{[0, h_{i+1,1}] \ [0, h_{j+1,1}]}^{\ \ \ \ \ \ \ \ \ \ \mathcal{U}(k)}\right)=\dim I_{L\left(\frac{1}{2},0\right)\otimes L\left(\frac{21}{22},0\right)}\left(_{[0, h_{i+1,1}] \ [0, h_{j+1,1}]}^{\ \ \ \ \ \ \ \ [0, h_{k+1,1}]}\right).
\]
On the other hand, by directly computation, we have
\[
 q\dim_{\mathcal{U}(0)}\mathcal{U}(i)=q\dim_{L\left(\frac{1}{2},0\right)\otimes L\left(\frac{21}{22},0\right)}[0, h_{i+1,1}]=\frac{\sin(\frac{(i+1)\pi}{11})}{\sin(\frac{\pi}{11})}.
\]
So we have
\[
\dim I_{\mathcal{U}(0)}\left(_{\mathcal{U}(i) \ \mathcal{U}(j)}^{\ \ \mathcal{U}(k)}\right)=\dim I_{L\left(\frac{1}{2},0\right)\otimes L\left(\frac{21}{22},0\right)}\left(_{[0, h_{i+1,1}] \ [0, h_{j+1,1}]}^{\ \ \ \ \ \ \ \ [0, h_{k+1,1}]}\right)=\dim I_{L\left(\frac{21}{22},0\right)}\left(_{L\left(\frac{21}{22},h_{i+1,1}\right) \ L\left(\frac{21}{22},h_{j+1,1}\right)}^{\ \ \ \ \ \ \ \ L\left(\frac{21}{22},h_{k+1,1}\right)}\right).
\].
Then we can conclude our theorem by using theorem \ref{fusion rules of virasoro modules}.
\end{proof}

\section*{Acknowledgments}
The author Wen Zheng thanks Xiangyu, Jiao for her discussions, helpful comments and the help in Latex.


\begin{thebibliography}{ADJR}
\bibitem[ABD]{ABD} T. Abe, G. Buhl, C. Dong, Rationality, regularity
and $C_{2}$-cofiniteness. \emph{Trans. AMS.} \textbf{356} (2004),
3391\textendash 3402.

\bibitem[ADJR]{ADJR} C. Ai, C. Dong, X. Jiao, L. Ren, The irreducible
modules and fusion rules for the parafermion vertex operator algebras.
\emph{Trans. Amer. Math. Soc. }\textbf{370} (2018), no. 8, 5963\textendash 5981.

\bibitem[ADL]{ADL} T. Abe, C. Dong, H. Li, Fusion rules for the vertex
operator algebras $M(1)^{+}$ and $V_{L}^{+}$ . \emph{Comm. Math.
Phys.} \textbf{253} (2005), 171\textendash 219.

\bibitem[C]{C} J. H. Conway, A simple construction for the Fisher-Griess
Monster group. \emph{Invent. Math. }\textbf{79 }(1985), 513\textendash 540.

\bibitem[DJX]{DJX} C. Dong, X. Jiao, F. Xu, Quantum dimensions and
quantum Galois theory. \emph{Trans. Amer. Math. Soc.}\textbf{ 365}
(2013), 6441\textendash 6469.

\bibitem[DJY]{DJY} C.Dong, X.Jiao, N.Yu, 6A-Algebra and its representations.
 \emph{J. Algebra }\textbf{553 }(2019),
174\textendash 210.

\bibitem[DLM1]{DLM1} C. Dong, H. Li, G. Mason, Regularity of rational
vertex operator algebras. \emph{Adv. Math.} \textbf{132} (1997), 148\textendash 166.

\bibitem[DLM2]{DLM2} C. Dong, H. Li, G. Mason, Twisted representations
of vertex operator algebras. \emph{Math. Ann. }\textbf{310 }(1998),\textbf{
}571\textendash 600.

\bibitem[DLM3]{DLM3}C. Dong, H. Li, G. Mason, Simple Currents and
Extensions of Vertex Operator Algebras. \emph{Comm. Math. Phys.} \textbf{180
}(1996), 671-707.



\bibitem[DM1]{DM1} C.Dong, G.Mason, Rational vertex operator algebras and the effective central charge.
 \emph{International Math. Research Notices }\textbf{56 }(2004),
2989\textendash 3008.

\bibitem[DM2]{DM2} C. Dong, G. Mason, On quantum Galois theory.\emph{
Duke Math. J. }\textbf{86} (1997), 305\textendash 321.


\bibitem[DMZ]{DMZ} C. Dong, G. Mason and Y. Zhu, Discrete series of the
Virasoro algebra and the moonshine module, {\it Proc. Symp. Pure.
Math., AMS} {\bf 56} II (1994), 295--316.

\bibitem[G]{G} R. Griess, "The friendly giant."\emph{
Invent. Math. }\textbf{69} (1982), 1\textendash 102.

\bibitem[FFK]{FFK} G. Felder, J. Fr\"{o}hlich, G. Keller, Braid matrices
and structure constants for minimal conformal models. \emph{Comm.
Math. Phys.} \textbf{124} (1989), no. 4, 647\textendash 664.

\bibitem[FHL]{FHL} I. B. Frenkel, Y. Huang, J. Lepowsky, On axiomatic
approaches to vertex operator algebras and modules. \emph{Memoirs
American Math. Soc. }\textbf{104}, 1993.

\bibitem[FLM]{FLM} I. B. Frenkel, J. Lepowsky, A. Meurman, Vertex
operator algebras and the monster. \emph{Pure and Applied Math., }vol.
134, Academic Press, Massachusetts, 1988.

\bibitem[H1]{H1} Y. Huang, Virasoro vertex operator algebras, the
(nonmeromorphic) operator product expansion and the tensor product
theory. \emph{J. Algebra} \textbf{182} (1996), 201\textendash 234.

\bibitem[H2]{H2} Y. Huang, Generalized rationality and a ``jacobi
identity\textquotedblright{} for intertwining operator algebras. \emph{Selecta
Math.}\textbf{ 6 }(2000), 225\textendash 267 .

\bibitem[KZ]{KZ} V.G. Knizhnik, A.B. Zamolodchikov, Current algebras
and wess-zumino model in two dimensions.\emph{ Nucl. Phys. B.} \textbf{247}
(1984), 83\textendash 103.

\bibitem[L]{L}H. Li, Symmetric invariant bilinear forms on vertex
operator algebras. \emph{J. Pure Appl. Algebra} \textbf{96 }(1994),
279\textendash 297.

\bibitem[LYY]{LYY} C. H. Lam, H. Yamada, H. Yamauchi, McKay's observation
and vertex operator algebras generated by two conformal vectors of
central charge $\frac{1}{2}$. \emph{Int. Math. Res. pap. }(2005),
no. 3, 117\textendash 181.

\bibitem[LYY1]{LYY1} C. H. Lam, H. Yamada, H. Yamauchi, Vertex operator
algebras, extended $E_{8}$ diagram, and McKay's observation on the
monster simple group\emph{. Trans. Amer. Math. Soc. }\textbf{359}
(2007), no. 9, 4107\textendash 4123.

\bibitem[M]{M} M. Miyamoto, Griess algebras and conformal vectors
in vertex operator algebras. \emph{J. Algebra }\textbf{179 }(1996),
523\textendash 548.

\bibitem[S]{S} S. Sakuma, $6$-Transposition property of $\tau$-involutions
of vertex operator algebras. \emph{Int. Math. Res. Not. }\textbf{2007}
(2007), no. 9. rnm 030, 19pp.

\bibitem[SY]{SY} S. Sakuma, H. Yamauchi, Vertex operator algebra
with two Miyamoto involutions generating $S_{3}$. \emph{J. Algebra
}\textbf{267 }(2003), 272\textendash 297.

\bibitem[TK]{TK} A. Tsuchiya, Y. Kanie, Vertex operators in conformal
field theory on P1 and monodromy representations of braid group. \emph{Adv.
Stud. Pure Math.} \textbf{16} (1988) 297\textendash 372.

\bibitem[W]{W} W. Wang, Rationality of Virasoro vertex operator algebras\emph{.
Internat. Math. Res. Notices. }Yale University. Connecticut, 1990\emph{. }

\end{thebibliography}
\end{document}